\documentclass[11pt,authoryear,titlepage]{article}
\pdfoutput = 1
\usepackage{fullpage}
\usepackage{amsthm}
\usepackage{amsmath}
\usepackage{enumitem}
\usepackage{amssymb}
\usepackage{setspace}
\usepackage{graphicx}
\usepackage{amsfonts}
\usepackage{makecell}
\usepackage{algorithm}
\usepackage{algpseudocode}
\usepackage{booktabs}
\usepackage{soul}
\usepackage[hang,flushmargin]{footmisc}
\usepackage[usenames,dvipsnames]{xcolor}
\usepackage[colorlinks=TRUE,citecolor=blue,linkcolor=BrickRed,urlcolor=PineGreen]{hyperref}
\usepackage[left=2.5cm,right=2.5cm,top=3cm,bottom=3cm]{geometry}
\allowdisplaybreaks

\usepackage{natbib}
\newcommand*{\doi}[1]{\href{http://doi.org/#1}{doi: #1}}

\newcommand{\bs}[1]{\boldsymbol{#1}}

\DeclareMathOperator*{\argmax}{arg\ max}
\DeclareMathOperator*{\arcoth}{arcoth}
\DeclareMathOperator*{\Cauchy}{Cauchy}
\DeclareMathOperator*{\sign}{sign}

\def\E{\mathrm{E}}

\def\Pr{\mathrm{Pr}}

\DeclareMathOperator*{\N}{N}
\DeclareMathOperator*{\M}{M}

\floatname{algorithm}{Procedure}
\newtheorem{theorem}{Theorem}

\newtheorem{lemma}{Lemma}
\newtheorem{corollary}{Corollary}

\begin{document}
\singlespacing
\title{Constructing and extending $n=1$ Bayesian confidence intervals for location parameters in location-scale families}
\author{David Gerard$^1$
  \vspace{3mm}\\
  \small$^1$Department of Mathematics and Statistics, American University, Washington, DC, 20016, USA}
\date{}
\maketitle

\begin{abstract}
  It is a surprising, modestly known fact that when given a single
  observation from a normal distribution with unknown mean and unknown
  variance, valid and non-trivial confidence intervals for the mean
  can be constructed. These intervals are presented in papers fully
  formed, providing limited intuition for how they arise or how to
  generalize them. We show that these intervals can be constructed in
  a principled way using two separate Bayesian reasonings. In the
  first, for any continuous symmetric location-scale family (under
  mild regularity conditions) with $n=1$ observation, we derive priors
  which produce $(1 - \alpha)100\%$ credible intervals that are,
  asymptotically in the confidence level $\alpha \rightarrow 0$, valid
  $(1 - \alpha)100\%$ confidence intervals. In the second, we show
  that the $n=1$ frequentist intervals can be seen as $t$-intervals
  augmented with a prior value, and that these augmented $t$-intervals
  are equivalent to inverted frequentist tests using a Bayes factor
  (using appropriate priors) as a test statistic. For $n \geq 2$, our
  credible interval approach does not maintain the confidence
  level. However, for $n \geq 2$, our augmented $t$-intervals produce
  valid confidence intervals with lower expected squared width in
  parts of the parameter space than the Student $t$-intervals,
  indicating improvements when prior knowledge is available. We
  demonstrate these methods on an $n = 3$ dataset of hyperbolic excess
  velocities of interstellar objects.
\end{abstract}

\section{Introduction}
\label{sec:problem.statement}

Suppose we have a single observation from a normal distribution
\begin{align}
  \label{eq:normal.obs}
  X \sim \N(\mu, \sigma^2),
\end{align}
where neither $\mu$ nor $\sigma^2$ is known. Surprisingly, valid
$(1 - \alpha)100\%$ confidence intervals for $\mu$ are possible, in
that they have minimum coverage of $1 - \alpha$. This was first
documented by \citet{abbott1962two}, and has enjoyed some steady
interest since \citep[e.g.,][]{blachman1987confidence,
  edelman1990confidence, rodriguez1996confidence, wall2001effective,
  portnoy2019invariance}.  There are a few forms these intervals take
\citep{blachman1987confidence}, but the one we will characterize using
Bayesian reasoning is
\begin{align}
  \label{eq:ci.init.2}
  \frac{X+A}{2} &\pm \eta |X - A|,
\end{align}
for some pre-specified value $A$. For a given $\alpha \leq 1/2$, it is
possible to choose a large enough value of $\eta$ to bound the
coverage probability of \eqref{eq:ci.init.2} below by $1 - \alpha$,
thus providing valid (though typically conservative) confidence
intervals
(Appendix~\ref{sec:bm.deriv}). \citet{blachman1987confidence} provide
a very nice treatment of this.

Where does \eqref{eq:ci.init.2} come from? The typical paper in this
line of literature first presents \eqref{eq:ci.init.2} and then
derives its properties. We wanted to derive \eqref{eq:ci.init.2} from
inferential principles. To do so, in this manuscript we provide two
Bayesian constructions for interval \eqref{eq:ci.init.2}. In the
first, we derive priors that yield credible intervals that are
approximately equal to \eqref{eq:ci.init.2}. This approximation is
asymptotic \emph{not} in the sample size (which stays $n = 1$) but in
the confidence level --- that is, for $\alpha$ close to 0, which is
the typical case. Interestingly, this construction works for any
continuous symmetric location-scale family (under mild regularity
conditions), and so we present it in this more general form. In the
second, we show that \eqref{eq:ci.init.2} can be constructed by
inverting a frequentist test that uses a Bayes factor (BF)
\citep{kass1995bayes} (using appropriate priors) as a test statistic.

How do we extend \eqref{eq:ci.init.2} for $n \geq 2$? Of course, for
$n \geq 2$, Student $t$-intervals are the workhorse of mean interval
estimation \citep{student1908probable}. But these only have optimality
guarantees among certain classes of confidence intervals (e.g.,
unbiased or equivariant) \citep{lehmann2006testing}. There has been
some relatively recent interest in creating biased/non-equivariant
intervals with improved performance in parts of the parameter space
that a researcher \emph{a priori} considers more likely to contain the
mean. \citet{wall2001effective} note that since
$\bar{X} \sim \N(\mu,\sigma^2/n)$ one can use \eqref{eq:ci.init.2}
with $\bar{X}$ instead of $X$. They show that there are regions of the
parameter space when $n=2$ (but not $n>2$) where this strategy results
in confidence intervals with smaller expected width than the Student
$t$-intervals. \citet{yu2018adaptive} develop frequentist assisted by
Bayes (FAB) intervals with minimum posterior expected width among a
certain class of intervals with constant coverage, but this class is
not exhaustive, making improvements in parts of the parameter space
possible. Our Bayesian paradigms for constructing \eqref{eq:ci.init.2}
provide natural generalizations for $n \geq 2$. We show that using the
same priors for $n=1$ does not produce valid confidence intervals for
$n\geq 2$ in our credible interval approach. However, we show that our
BF intervals have lower expected squared width in parts of the
parameter space than the Student $t$-interval for all $n \geq 2$,
indicating a non-trivial improvement over the Student $t$-intervals
when prior knowledge is available. We demonstrate these intervals on a
dataset of hyperbolic excess velocities of interstellar objects, where
$n=3$.

The remainder of the paper is organized as
follows. Section~\ref{sec:gen.case} develops the credible interval
approach to approximate confidence intervals for $n=1$, and
Sections~\ref{sec:normal.family} and~\ref{sec:cauchy.family} apply it
to the normal and Cauchy cases,
respectively. Section~\ref{sec:bf.invert} constructs ``augmented
$t$-intervals'' by inverting a test based on Bayes factors. We then
turn to $n \geq 2$: Section~\ref{sec:n.2.cred.bad} shows that the
credible interval approach does not generalize directly, while
Section~\ref{sec:aug.t.int.gen} shows that the augmented $t$ approach
does. Section~\ref{sec:iso} presents our real data application, and
Section~\ref{sec:discussion} concludes. We defer all technical proofs
to the Supplementary Material (unless otherwise noted, in
Appendix~\ref{sec:proofs}).

\section{$n = 1$ Bayesian confidence intervals}
\label{sec:cred.approach}

\subsection{Credible intervals: symmetric location-scale family}
\label{sec:gen.case}

We will begin by constructing a prior for $(\mu,\sigma^2)$ whose
credible interval for $\mu$ is (asymptotically in the confidence
level) equivalent to the confidence interval \eqref{eq:ci.init.2}. Our
treatment in this section is actually valid for \emph{any} continuous
symmetric location-scale family (under mild regularity conditions),
and so we present it in this generality. Our strategy will be to first
parameterize the problem in terms of the deviation from the mean,
$\beta = \mu - A$, and the standardized deviation from the mean,
$\nu = |\mu - A| / \sigma$. We will then place a scale-sign invariant
prior over $\beta$ and choose the prior over $\nu$ that results in the
marginal posterior distribution for $\beta$ having the fattest
tails. This posterior distribution will be shown to have the same
tails as the confidence distribution \citep{xie2013confidence,
  schweder2016confidence} implied by the interval
\eqref{eq:ci.init.2}. Recall that a confidence distribution of a
parameter is a distribution where the $\alpha/2$ and $1 - \alpha/2$
quantiles form a valid $(1 - \alpha)100\%$ confidence interval. That
they have the same tails shows that the equal tailed credible
intervals using these priors are confidence intervals asymptotically
in the confidence level ($\alpha \rightarrow 0$).

We begin with the likelihood. Let $\rho(\cdot)$ be a known symmetric
density. Then the density of $X$ is
$f(x|\mu,\sigma^2) = \frac{1}{\sigma}\rho\left(\frac{x -
    \mu}{\sigma}\right)$. However, it will be easier to write this
density in terms of $Y = X - A$, $\beta = \mu - A$, and
$\nu = |\mu - A| / \sigma$, the density of which comes from
Lemma~\ref{lem:y.dense}.

\begin{lemma}
  \label{lem:y.dense}
  The density of $Y$ is
  \begin{align}
    \label{eq:likelihood.beta.nu}
    f(y|\beta, \nu) &= \frac{\nu}{|\beta|}\rho\left(\nu\frac{y}{\beta} - \nu\right).
  \end{align}
\end{lemma}

For the rest of this section, we will develop credible intervals for
$\beta$ given $Y$ which are approximate confidence intervals. Through
back-transforming, this results in credible intervals for $\mu$ given
$X$ which are approximate confidence intervals. That is, if
$[L(Y),U(Y)]$ is a $(1 - \alpha)100\%$ credible interval for $\beta$, then
the corresponding $(1 - \alpha)100\%$ credible interval for $\mu$ is just
$[L(X - A) + A,U(X - A) + A]$.

The invariance structure of the problem suggests a natural starting
point for the prior over $\beta$. For a known $\nu$, constructing
confidence interval \eqref{eq:ci.init.2} for $\beta$ can be seen as an
invariant decision problem using the multiplicative group of nonzero
reals. For $g \in \mathbb{R}\setminus\{0\}$, $Y \mapsto gY$,
$\beta \mapsto g\beta$, and
$Y/2 \pm \eta|Y| \mapsto g[Y/2 \pm \eta|Y|]$. There is a well-known
correspondence between invariance and Bayesian reasoning. Given an
invariant loss function, for an invariant decision problem under a
group acting transitively on the parameter, the best equivariant
decision rule coincides with the (generalized) Bayes rule under the
right invariant (generalized) Haar prior \citep[e.g.,][Section
6.6]{berger2013statistical}. All of this suggests that we begin by
placing the right invariant generalized Haar prior for the
multiplicative group on $\beta$ (which would be optimal in some sense
if $\nu$ were known). We will also assume $\beta$ and $\nu$ are
\emph{a priori} independent to get prior
\begin{align}
  \label{eq:prior.beta.nu}
  \pi(\beta,\nu) = |\beta|^{-1}\pi(\nu).
\end{align}
We will specify $\pi(\nu)$ later. But first, we have
Theorem~\ref{theo:post.ni1} which shows that the posterior
distribution of $1 / \beta$ given $\nu$ is in the same location-scale
family as $Y$ (with location $1 / Y$ and scale $1 / (\nu|Y|)$), and
the marginal posterior of $\nu$ is equal to the prior of $\nu$.

\begin{theorem}
  \label{theo:post.ni1}
  Given likelihood \eqref{eq:likelihood.beta.nu} and prior
  \eqref{eq:prior.beta.nu}, the posterior density of $\beta$ and $\nu$
  given $Y$ is
  \begin{align}
    \label{eq:ls.post.full}
    \pi(\beta,\nu|Y) &= \pi(\beta|\nu,Y)\pi(\nu|Y), \text{ where}\\
    \label{eq:ls.post}
    \pi(\beta|\nu,Y) &= \frac{\nu|Y|}{\beta^2}\rho\left(\nu\frac{Y}{\beta}-\nu\right), \text{ and}\\
    \pi(\nu|Y) &= \pi(\nu).
  \end{align}
\end{theorem}

We now need to choose a prior over $\nu$. Because of
Theorem~\ref{theo:post.ni1}, $\nu$ must have a proper prior, because
otherwise the posterior is not integrable. We will derive the prior
for $\nu$ based on the tail probabilities from
Theorem~\ref{theo:tail.prob}.
\begin{theorem}
  \label{theo:tail.prob}
  Let $\rho(\cdot)$ be differentiable at $\nu$. Then the
  conditional posterior tail probability of $\beta$ given $\nu$ and
  $Y$ is
  \begin{align}
    \label{eq:tail.prob}
    \Pr(\beta > z|\nu,Y) = \frac{|Y|}{z}\nu\rho\left(\nu\right) + \mathcal{O}\left(\frac{1}{z^2}\right).
  \end{align}
\end{theorem}

Theorem~\ref{theo:tail.prob} gives the tail probability conditional on
$\nu$. For some location-scale families, the marginal tail
probability is the integral over this conditional tail probability for
\emph{any} choice of proper prior $\pi(\cdot)$ over $\nu$,
\begin{align}
  \label{eq:marg.tail.prob.text}
\Pr(\beta > z|Y) &= \frac{|Y|}{z}\int_{0}^{\infty}\nu\rho\left(\nu\right)\pi(\nu)\mathrm{d}\nu + \mathcal{O}\left(\frac{1}{z^2}\right).
\end{align}
Sufficient regularity conditions on $\rho(\cdot)$ for
\eqref{eq:marg.tail.prob.text} to hold for any proper prior
$\pi(\cdot)$ are provided in Appendix~\ref{sec:marg.tail}. Namely, for
densities that are differentiable on $\mathbb{R}$, sufficient
conditions are
\begin{align}
  \label{eq:reg.cond.cont}
  \sup_{u \geq 0} u\rho(u) < \infty, \text{ and } \sup_{u \geq 0} u^2|\rho'(u)| < \infty.
\end{align}
In particular, if $\rho(\cdot)$ is the standard normal density, it
satisfies \eqref{eq:reg.cond.cont} and any proper prior $\pi(\cdot)$
results in marginal tail approximation \eqref{eq:marg.tail.prob.text}
(Theorem~\ref{theo:normal.reg.cond}).

However, if we restrict $\pi(\cdot)$ to be a point mass prior at $\nu_0$, then
\eqref{eq:marg.tail.prob.text} holds under a much weaker condition:
$\rho(\cdot)$ need only be differentiable at the single
point $\nu_0$, rather than satisfying the global conditions of
Appendix~\ref{sec:marg.tail} (Corollary~\ref{cor:pointmass.marg}).
\begin{corollary}
  \label{cor:pointmass.marg}
  Let $\pi(\nu) = \delta_{\nu_0}(\nu)$, a point mass distribution at
  $\nu_0 > 0$, and let $\rho(\cdot)$ be differentiable at
  $\nu_0$. Then
  \begin{align}
    \Pr(\beta > z|Y) = \frac{|Y|}{z}\nu_0\rho\left(\nu_0\right) + \mathcal{O}\left(\frac{1}{z^2}\right).
  \end{align}
\end{corollary}
\begin{proof}
  Under a point mass prior the marginal tail probability equals the
  conditional tail probability, so the result is immediate from
  Theorem~\ref{theo:tail.prob}.
\end{proof}

Based on Theorem~\ref{theo:tail.prob}, we can obtain the fattest tails
for the marginal posterior distribution of $\beta$ by having our prior
for $\nu$ be a point mass at the value that maximizes
$\nu\rho\left(\nu\right)$. At least, it will be the fattest tails
among the class of point mass priors over $\nu$
(Corollary~\ref{cor:pointmass.marg}). But for some $\rho(\cdot)$ (like
the normal), it will be the fattest tails over the entire class of
proper priors over $\nu$ (Appendix~\ref{sec:marg.tail}).  This
indicates that the marginal posterior tails for $\beta$, using this
point mass prior on $\nu$, are
\begin{align}
  \label{eq:post.tails}
  \Pr(\beta > z|Y) = \frac{|Y|}{z}\max_{\nu>0}\nu\rho\left(\nu\right) + \mathcal{O}\left(\frac{1}{z^2}\right).
\end{align}
This leads us to the main theorem of this section
(Theorem~\ref{theo:valid.conf.post}).
\begin{theorem}
  \label{theo:valid.conf.post}
  Assume we have likelihood \eqref{eq:likelihood.beta.nu} and prior
  \begin{align}
    \label{eq:prior.best}
    \pi(\beta,\nu) &= |\beta|^{-1}\delta_{\tilde{\nu}}(\nu), \text{ where}\\
    \label{eq:prior.nu.tilde}
    \tilde{\nu} &= \argmax_{\nu>0}\nu\rho\left(\nu\right),
  \end{align}
  and $\delta_{\tilde{\nu}}(\nu)$ is a point mass at
  $\tilde{\nu}$. Furthermore, assume that $\rho(\cdot)$ is
  differentiable at $\tilde{\nu}$.  Then the equal tailed
  $(1 - \alpha)100\%$ credible interval for $\beta$ has confidence
  error probability $\alpha + \mathcal{O}(\alpha^2)$.
\end{theorem}

\begin{enumerate}[label=\textit{Note \arabic*.}, leftmargin=*, wide=0pt]
\item Recall the ``mild regularity condition'' on $\rho(\cdot)$ that
  we said is required to build credible intervals that are
  asymptotically valid confidence intervals. That condition is that
  $\rho(\cdot)$ is differentiable at $\tilde{\nu}$. To demonstrate the
  issue that can occur when $\rho(\cdot)$ is discontinuous at
  $\tilde{\nu}$, consider the uniform density with support on
  $[-0.5, 0.5]$, which is discontinuous at $-\tilde{\nu} = -0.5$ and
  $\tilde{\nu} = 0.5$. One can show that, using prior
  \eqref{eq:prior.best}, the posterior density is
  \begin{align}
    \pi(\beta|Y) =
    \begin{cases}
      \frac{|Y|}{2\beta^2}1\left(\beta \geq \frac{Y}{2}\right) &\text{ if } Y > 0\\
      \frac{|Y|}{2\beta^2}1\left(\beta \leq \frac{Y}{2}\right) &\text{ if } Y < 0.
    \end{cases}
  \end{align}
  If $Y > 0$, the upper limit of the credible interval is the same as
  the upper limit of the confidence interval \eqref{eq:ci.init.2} (up
  to order $\mathcal{O}(\alpha^2)$), but the lower limits
  disagree. Likewise, if $Y < 0$ then the lower limits agree (up to
  order $\mathcal{O}(\alpha^2)$) but the upper limits disagree. The
  issue arises because the application of Taylor's theorem in the
  proof of Theorem~\ref{theo:valid.conf.post} requires
  differentiability at $\tilde{\nu}$.
\item We can actually relax the differentiability requirement at
  $\tilde{\nu}$ in Theorem~\ref{theo:valid.conf.post} to one of
  continuity at $\tilde{\nu}$. However, this results in the remainder
  being $\mathrm{o}\left(\alpha\right)$ instead of
  $\mathcal{O}\left(\alpha^2\right)$ (use the Peano form of the
  remainder for the first-order instead of the second-order Taylor
  series). We need differentiability to get a quadratic-order
  approximation. Of course, this relaxation does not save the uniform
  distribution.
\item The regularity conditions in Appendix~\ref{sec:marg.tail} are
  not necessary for Theorem~\ref{theo:valid.conf.post}, which rests
  only on Corollary~\ref{cor:pointmass.marg} and so holds whenever
  $\rho(\cdot)$ is differentiable at $\tilde{\nu}$. What is at issue
  in Appendix~\ref{sec:marg.tail} is instead the ability to interpret
  prior \eqref{eq:prior.best} as the one producing the fattest
  posterior tails among the class of proper priors. This is a useful
  interpretation, but not one the theorem requires.
\end{enumerate}

\subsection{Credible intervals: normal family}
\label{sec:normal.family}

Let us apply the results from Section~\ref{sec:gen.case} to the normal
case \eqref{eq:normal.obs}. Since
$\nu\rho(\nu) = (2\pi)^{-1/2}\nu\exp(-\nu^2/2)$ is maximized at
$\nu = 1$, we \emph{a priori} assume $\sigma^2 = (\mu - A)^2$ and we
use the following model
\begin{align}
  \label{eq:x.coef.is.1}
  X &\sim \N(\mu, (\mu - A)^2),\\
  \label{eq:mu.prior.coef.is.1}
  \pi(\mu) &= |\mu - A|^{-1}.
\end{align}
Building credible intervals assuming
\eqref{eq:x.coef.is.1}--\eqref{eq:mu.prior.coef.is.1} results in valid
confidence intervals (asymptotically in the confidence level) for the
model $X \sim \N(\mu, \sigma^2)$. Though, again it is easier to work
with $Y = X - A$ and $\beta = \mu - A$, so we can write
\eqref{eq:x.coef.is.1}--\eqref{eq:mu.prior.coef.is.1} as
\begin{align}
  \label{eq:y.coef.is.1}
  Y &\sim \N(\beta, \beta^2),\\
  \label{eq:beta.prior.coef.is.1}
  \pi(\beta) &= |\beta|^{-1},
\end{align}
and we note that our results carry through for $X$ and $\mu$.

The posterior distribution for $\beta$, assuming
\eqref{eq:y.coef.is.1}--\eqref{eq:beta.prior.coef.is.1}, is ``inverse
normal'' \citep{robert1991generalized} (which is different from the
inverse Gaussian), with inverse mean $1/Y$ and inverse variance
$1/Y^2$. Figure~\ref{fig:post.dens.graphic} shows the posterior density
on the original scale of $\mu$ given $X$, for the case $A > X$.
Reflecting the figure horizontally about $(X + A)/2$ gives the
posterior density when $A < X$. The posterior distribution of $(\mu -
A)/|X - A|$ is a fixed distribution (inverse normal with inverse mean
and inverse variance $1$), so Figure~\ref{fig:post.dens.graphic}
applies to any value of $A > X$. The distribution has two modes: using
\citet{robert1991generalized}, the larger is at $(X + A)/2$ and the
smaller is at $A + |X - A|$. Exactly half of the posterior area lies
between $X$ and $A$, about 16\% lies above $A$, and about 34\% lies
below $X$.

\citet{robert1991generalized} does not list the CDF and quantile
functions of this distribution, but they are easy to derive and we do
so in Appendix~\ref{sec:inv.norm.properties} (along with the inverse
of any continuous location-scale family). Using
\eqref{eq:quant.invnorm}, the posterior median of $\beta$ is
\begin{align}
\frac{Y}{1 - \Phi^{-1}\left(\Phi(1) - 1/2\right)} \approx 0.71 Y.
\end{align}
Using \eqref{eq:quant.invnorm} to obtain $\alpha/2$ and $1 - \alpha/2$
quantiles, for $\alpha < 2(1 - \Phi(1)) \approx 0.3173$, we obtain a
$(1 - \alpha)100\%$ credible interval for $\beta$. If $Y>0$, this
interval is
\begin{align}
  \label{eq:normal.cred.interval}
  \left(\frac{Y}{1 - \Phi^{-1}\left(\Phi(1) + \frac{\alpha}{2}\right)}, \frac{Y}{1 - \Phi^{-1}\left(\Phi(1) - \frac{\alpha}{2}\right)}\right),
\end{align}
and the bounds swap for $Y < 0$. For $\alpha = 0.05$, this corresponds
to about $(-9.15Y,10.15Y)$, or in the original variables
$(-9.15(X-A) + A,10.15(X - A) + A)$.

Interval \eqref{eq:normal.cred.interval} is scale-sign invariant
\citep{portnoy2019invariance}, and so we can use a corrected Theorem
1.1 of \citet{portnoy2019invariance} to calculate the actual coverage
probability of our credible interval. Suppose we have an interval of
the form $(c_1Y,c_2Y)$, then Theorem 1.1 of
\citet{portnoy2019invariance} states that the level of the interval
for a given $\nu$ is
\begin{align}
    \label{eq:bayes.true.level}
  \text{True Level} =
  \begin{cases}
    \Phi(\nu(1 - 1/c_2)) + 1 - \Phi(\nu(1 - 1 / c_1)) & \text{ if } c_1 \leq 0 \text{ and } c_2 \geq 0\\
    \Phi(\nu(1 - 1/c_2)) - \Phi(\nu(1 - 1 / c_1)) & \text{ if } c_1 > 0 \text{ and } c_2 > 0.
  \end{cases}
\end{align}
Note that Theorem 1.1 of \citet{portnoy2019invariance} has a typo ---
it has $1+1/c_1$ everywhere instead of $1-1/c_1$. In our case, based
on \eqref{eq:normal.cred.interval}, we have, for
$\alpha < 2(1 - \Phi(1)) \approx 0.3173$,
\begin{align}
  c_1 = \left[1 - \Phi^{-1}\left(\Phi(1) + \frac{\alpha}{2}\right)\right]^{-1} \text{ and }  c_2 = \left[1 - \Phi^{-1}\left(\Phi(1) - \frac{\alpha}{2}\right)\right]^{-1}.
\end{align}

Table~\ref{tab:compare.bayes} compares the intervals
\eqref{eq:ci.init.2} and our new Bayesian credible intervals for
various levels of $\alpha$. At $\alpha \leq 0.1$, they are within 2
decimal places of each other.  Indeed the 95\% credible interval using
the Bayes procedure is at least 2 decimal points within the best
possible confidence interval (in terms of a minimax criterion)
\citep{portnoy2019invariance}. We also numerically minimized
\eqref{eq:bayes.true.level} over $\nu$ to obtain the worst error
probability of the Bayesian credible interval. For $\alpha = 0.05$,
the error probability of the Bayesian credible interval is within
$10^{-6}$ of 0.05 which, for all practical purposes (though not
mathematically exactly), means that it maintains the proper confidence
level.

\subsection{Credible intervals: Cauchy family}
\label{sec:cauchy.family}

Suppose $\rho(\cdot)$ is standard Cauchy
\begin{align}
  \rho(x) = \frac{1}{\pi(1 + x^2)}.
\end{align}
In this section, we show that the $n = 1$ posterior credible intervals
using prior \eqref{eq:prior.best} are \emph{exact} (not asymptotic)
confidence intervals for any confidence level at least $1/2$
(Corollary~\ref{cor:cauchy.exact}). We will work with the transformed
problem $Y = X - A$ and $\beta = \mu - A$, noting that the results
immediately carry over to $X$ and $\mu$. We will first derive the
confidence distribution of $\beta$ implied by the interval
$\frac{Y}{2} \pm \eta |Y|$
(Theorem~\ref{theo:confidence.dist.cauchy}), then show that it exactly
matches the marginal posterior distribution of $\beta$ given $Y$
(Theorem~\ref{theo:posterior.dist.cauchy}). Namely, letting
$Z = \frac{\beta - Y/2}{|Y|/2}$, then conditional on $Y$, we show that
$Z$ is standard Cauchy for both the posterior and for the confidence
distribution (for $|Z| \geq 1$). The confidence distribution does not
exist for $|\eta| < 1/2$ ($|Z| < 1$), so the results hold for
$\alpha \leq 1/2$ (the typical case).

\begin{theorem}
  \label{theo:confidence.dist.cauchy}
  Let $Y \sim \Cauchy(\beta,|\beta|/\nu)$.  Let $\beta$ follow the
  confidence distribution implied by the $(1-\alpha)100\%$ confidence
  interval $Y/2 \pm \eta |Y|$, where $\eta \geq \frac{1}{2}$ is
  defined as in Appendix~\ref{sec:bm.deriv} to control the confidence
  error. Let $Z = \frac{\beta - Y/2}{|Y|/2}$ (here, $\beta$ is random
  and $Y$ is fixed). Then the confidence distribution of $Z$ agrees
  with the $\Cauchy(0,1)$ law for $|z| \geq 1$.
\end{theorem}

\begin{theorem}
  \label{theo:posterior.dist.cauchy}
  Let $Y \sim \Cauchy(\beta,|\beta|/\nu)$.  Let $\beta$ follow the
  marginal posterior distribution given prior
  \eqref{eq:prior.best}. Let $Z = \frac{\beta - Y/2}{|Y|/2}$. Then
  $Z | Y \sim \Cauchy(0,1)$.
\end{theorem}

\begin{corollary}
  \label{cor:cauchy.exact}
  Let $Y \sim \Cauchy(\beta,|\beta|/\nu)$. Let the prior for $\beta$
  be \eqref{eq:prior.best}. Then every equal tailed
  $(1 - \alpha)100\%$ credible interval for $\beta$ is an \emph{exact}
  $(1 - \alpha)100\%$ confidence interval for $\alpha \leq 1/2$.
\end{corollary}
\begin{proof}
  Theorems~\ref{theo:confidence.dist.cauchy} and
  \ref{theo:posterior.dist.cauchy} hold for $\eta \geq 1/2$. For
  $\eta = 1/2$, the confidence level of interval $Y/2 \pm \eta |Y|$
  is exactly 1/2 \citep{blachman1987confidence}.
\end{proof}

\subsection{Inverting a Bayes test}
\label{sec:bf.invert}

The elementary link between frequentist hypothesis testing and
confidence intervals \citep[e.g.,][]{casella2024statistical} is weaker
in the Bayesian paradigm. That is, credible intervals are constructed
directly from a posterior distribution, not via inverting some Bayes
test \citep{gelman2013bayesian}. However, one approach to developing a
Bayesian analogue of \eqref{eq:ci.init.2} is to develop a Bayes test
that is equivalent to the test implied by \eqref{eq:ci.init.2}. We can
then invert this Bayes test to get an equivalent confidence
interval. The standard Bayesian approach to hypothesis testing is to
use Bayes factors \citep{kass1995bayes}, and so we will focus on
inverting a frequentist test that uses a Bayes factor as a test
statistic. This might seem like an odd approach, but there are some
recent frequentist justifications for using the Bayes factor as a test
statistic in this way \citep{bayarri2016rejection,
  fowlie2023neyman}. For an alternative link between the frequentist
confidence interval and Bayes testing approaches, see
Appendix~\ref{sec:pdr} where we discuss posterior density ratio tests
\citep{basu1996bayesian}.

The CI in \eqref{eq:ci.init.2} can be seen as an inverted
``augmented'' $t$-test, where the $t$-statistic is calculated from the
augmented data $X$ and $A$ (treating $A$ as an observation). This is
because $(X + A)/2$ and $|X-A|/\sqrt{2}$ are the sample mean and sample
standard deviation, respectively, of $X$ and $A$. Thus, any Bayes
factor which is a monotonic function of this augmented $t$-statistic
can be used as a test statistic which results in a test equivalent to
that implied by \eqref{eq:ci.init.2}.  In the basic (non-augmented)
case, many Bayes factors have been found which are elementary
functions of the one-sample and two-sample $t$-statistics
\citep{jeffreys1998theory, gonen2005bayesian, rouder2009bayesian,
  wang2016simple, gonen2019comparing, gronau2020informed,
  francis2023equivalent}. We have the following theorem for the
augmented $t$-statistic.

\begin{theorem}
  \label{theo:aug.t.as.bf}
  Let $X_1,\ldots,X_n \overset{\text{iid}}{\sim} N(\mu,
  \sigma^2)$. Suppose we are testing $H_0: \mu = \mu_0$ versus
  $H_A: \mu \neq \mu_0$. Let $\delta = (\mu - \mu_0)/\sigma$ be the
  standardized effect size. Consider the prior under the null,
  \begin{align}
    \label{eq:prior.h0.aug}
    \pi(\sigma^2) = \frac{\frac{1}{\sigma^2}N(A|\mu_0, \sigma^2)}{\int_{\sigma^2}\frac{1}{\sigma^2}N(A|\mu_0, \sigma^2)\mathrm{d}\sigma^2}
  \end{align}
  and under the alternative,
  \begin{align}
    \label{eq:prior.ha.aug}
    \pi(\delta, \sigma^2) = \frac{\frac{1}{\sigma^2}N(A|\mu_0 + \delta\sigma, \sigma^2)\pi(\delta)}{\int_{\delta,\sigma^2}\frac{1}{\sigma^2}N(A|\mu_0 + \delta\sigma, \sigma^2)\pi(\delta)\mathrm{d}\delta \mathrm{d}\sigma^2}.
  \end{align}
  I.e., we use prior $\pi(\sigma^2) = 1 / \sigma^2$ under the null and
  $\pi(\delta,\sigma^2) = \pi(\delta)/\sigma^2$ under the alternative
  and update both with one observation of value $A$. Then the Bayes
  factor for these two hypotheses is
  \begin{align}
    \label{eq:bf.def}
    \mathrm{BF} = \frac{\int_{\delta}T_{n}(t|\sqrt{n+1}\delta)\pi(\delta)\mathrm{d}\delta}{T_{n}(t)},
  \end{align}
  where $T_{n}(\cdot)$ is the $t$-density with $n$ degrees of freedom,
  $T_{n}(\cdot|b)$ is the non-central $t$-density with $n$ degrees of
  freedom and noncentrality parameter $b$, and $t$ is the augmented
  $t$-statistic calculated using data $X_1,\ldots,X_n,A$:
  \begin{align}
    \label{eq:aug.t.def}
    t = \frac{\hat{\mu} - \mu_0}{\hat{\sigma} / \sqrt{n+1}},
  \end{align}
  where $\hat{\mu}$ and $\hat{\sigma}$ are the sample mean and sample
  standard deviation of $X_1,\ldots,X_n,A$.
\end{theorem}
\begin{proof}
  See Appendix~\ref{sec:bf.aug.t}. A proof of a similar result for
  two-sample $t$-tests was given by \citet{gronau2020informed}. But
  our proof is a little simpler because it avoids complicated
  integrals.
\end{proof}

The Bayes factor and the $t$-statistic can provide the same decision
for various thresholds when the prior over the standardized effect
size is symmetric (i.e., the Bayes factor is monotone increasing in
$|t|$; Theorem~\ref{thm:bf.monotone}). Although other texts have
proven the functional relationship between $t$ and the Bayes factor, we
do not think conditions for monotonicity have been proven before.

\begin{theorem}
  \label{thm:bf.monotone}
  Let $\mathrm{BF}(t)$ be the Bayes factor in \eqref{eq:bf.def} and $t$
  the augmented $t$-statistic in \eqref{eq:aug.t.def} using data
  $X_1,\ldots,X_n,A$. If $\pi(\delta)$ is symmetric about 0, then
  $\mathrm{BF}(t)$ is monotone increasing in $|t|$.
\end{theorem}

Theorem~\ref{thm:bf.monotone} implies that there exist two constants,
$c_1$ and $c_2$, such that intervals
\begin{align}
  \label{eq:standard.t.int}
  &\left\{\mu: \left|\frac{\hat{\mu} - \mu}{\hat{\sigma}/\sqrt{n+1}}\right| < c_1\right\} \text{ and}\\
  \label{eq:bf.equivalent.int}
  &\left\{\mu: \mathrm{BF}\left(\frac{\hat{\mu} - \mu}{\hat{\sigma}/\sqrt{n+1}}\right) < c_2\right\}
\end{align}
are equal. Interval \eqref{eq:standard.t.int} is the augmented
$t$-interval. Interval \eqref{eq:bf.equivalent.int} can be seen as the
values of $\mu$ such that evidence against $\mu$ (in the Bayesian
hypothesis testing sense) would not be as strong as some predetermined
value.

It is interesting to explore what Bayes factor thresholds in
\eqref{eq:bf.equivalent.int} correspond to different confidence levels
in \eqref{eq:standard.t.int}. For that, we need to specify the prior
$\pi(\delta)$. The most typical choice would likely be a Cauchy prior
on $\delta$ \citep{jeffreys1998theory}. We compare different augmented
$t$-statistic thresholds (and the corresponding level of the
confidence interval) to the corresponding Bayes factor using this
Cauchy prior in Table~\ref{tab:bf.t}. Perhaps unsurprisingly,
augmented $t$-statistics seem to be overly sure of the parameter,
relative to the level of evidence indicated by the Bayes factor
\citep{benjamin2019three}. E.g., if we were testing $H_0: \mu = \mu_0$
against $H_1: \mu \neq \mu_0$ and obtained a $p$-value of 0.01 using
the augmented $t$-statistic, this would only correspond to a Bayes
factor of 5.02, which provides ``positive'' but not ``strong''
evidence against the null using typical thresholds
\citep{kass1995bayes}.

\section{$n \geq 2$ extensions}

\subsection{A credible interval approach}
\label{sec:n.2.cred.bad}

In principle, generalizing the credible interval approach from
Section~\ref{sec:gen.case} is relatively straightforward: use prior
\eqref{eq:prior.best} and update with more data. In this section, we
will explore doing so, but demonstrate that the resulting credible
intervals do not control the confidence error probability.

In the normal case, using prior \eqref{eq:prior.best} results in the
model $Y_1,\ldots,Y_n \overset{\text{iid}}{\sim} \N(\beta, \beta^2)$ and
$\pi(\beta) = |\beta|^{-1}$, which yields the posterior
\begin{align}
  \label{eq:kernel.inv.norm.n2}
  \pi(\beta|\bs{Y}) = K\left(n+1,\frac{\sum_{i=1}^nY_i}{\sum_{i=1}^nY_i^2}, \frac{1}{\sum_{i=1}^nY_i^2}\right) |\beta|^{-(n+1)}\exp\left[ -\frac{1}{2/\sum_{i=1}^nY_i^2} \left(\frac{1}{\beta} - \frac{\sum_{i=1}^nY_i}{\sum_{i=1}^nY_i^2}\right)^2\right].
\end{align}
Equation \eqref{eq:kernel.inv.norm.n2} is a \emph{generalized} inverse
normal distribution \citep{robert1991generalized} with shape parameter
$n+1$, inverse mean parameter
$\frac{\sum_{i=1}^nY_i}{\sum_{i=1}^nY_i^2}$, and inverse variance
parameter $\frac{1}{\sum_{i=1}^nY_i^2}$. The $K(\cdot,\cdot,\cdot)$ in
\eqref{eq:kernel.inv.norm.n2} is the normalizing constant, calculated
by \citet{robert1991generalized}:
\begin{align}
  \label{eq:norm.const}
K(\alpha,\mu,\tau^2)^{-1} = \left[\tau^2\right]^{\frac{1}{2}(\alpha - 1)}\exp\left[-\frac{\mu^2}{2\tau^2}\right]2^{\frac{1}{2}(\alpha-1)}\Gamma\left[\frac{1}{2}(\alpha-1)\right] \M\left[\frac{1}{2}(\alpha-1),\frac{1}{2},\frac{\mu^2}{2\tau^2}\right],
\end{align}
where $\M(\cdot,\cdot,\cdot)$ is Kummer's confluent hypergeometric
function of the first kind \citep[\href{https://dlmf.nist.gov/13.2}{Section 13.2}]{NIST:DLMF}.

We did not find an implementation of the generalized inverse normal
distribution in R, and so we created one in our package. We also did
not find implementations of the confluent hypergeometric function
$\M(a,b,z)$ that were stable for arbitrary
$a \in \{1/2,1,3/2,2,\ldots\}$, $b = 1/2$, and $z > 0$, the values
possible for the posterior calculation in \eqref{eq:norm.const}. Thus,
we implemented one in our package. For small $z$, we use series
\href{https://dlmf.nist.gov/13.2.E2}{(13.2.2)} from
\citet{NIST:DLMF}. For large $z$, we use one of the recurrence
relations of $\M(\cdot,\cdot,\cdot)$ to obtain stable function
values. From starting values (derived from Mathematica code in the
Supplementary Materials) \citep{Mathematica}
\begin{align}
  \begin{split}
    \M(0,1/2,z) &= 1,\\
    \M(1/2,1/2,z) &= \exp(z),\\
    \M(1,1/2,z) &= 1 + \exp(z)\sqrt{\pi z}(2\Phi(\sqrt{2z}) - 1), \text{ and}\\
    \M(3/2,1/2,z) &= \exp(z) (1 + 2z),
  \end{split}
\end{align}
we use recurrence relation
\href{https://dlmf.nist.gov/13.3.E1}{(13.3.1)} from \citet{NIST:DLMF}:
\begin{align}
  \label{eq:dlmf.reccurance}
   a\M(a+1,b,z) = (b - a)\M(a - 1, b, z) + (2a - b + z)\M(a,b,z).
\end{align}
Since the $\M(\cdot,\cdot,\cdot)$ terms in \eqref{eq:dlmf.reccurance}
are all positive, we can keep all terms on the log-scale and update
using LogSumExp, which gives us greater numeric stability for large
$z$,
\begin{align}
  \begin{split}
  \log\M(a+1,b,z) &= \log\left[(b-a)e^{\log\M(a - 1, b, z) - w} + (2a-b+z)e^{\log\M(a,b,z) - w} \right] + w -\log(a), \text{ where}\\
    w &= \max[\log \M(a - 1, b, z), \log \M(a,b,z)].
  \end{split}
\end{align}
E.g., calculating $\M(2,1/2, 100)$ returns \texttt{NaN}'s when
using GSL \citep{galassi2009gnu}, but returns $\exp(107.5)$ when using
our recurrence method. We also implemented CDF and quantile functions
for the generalized inverse normal via numeric integration and
root-finding tools.

Posterior \eqref{eq:kernel.inv.norm.n2} does not produce credible
intervals that are valid confidence intervals. We simulated
$Y_1,\ldots,Y_n \overset{\text{iid}}{\sim} \N(\beta, \beta^2 / \nu^2)$ for
$\beta \in \{1, 5, 10\}$, $\nu \in \{0.5, 1, 2\}$, and
$n \in \{2, 10, 100\}$ for 1000 replications for each simulation
scenario. In each replication, we generated 95\% credible intervals
using posterior \eqref{eq:kernel.inv.norm.n2}, and calculated the
confidence error rate for the 1000 replications. These rates are
presented in Figure~\ref{fig:cred.t1e.n2}. We see there that only when
$\nu = 1$ do the credible intervals have the correct confidence error
rate. (The correct coverage at $\nu = 1$ is to be expected given our
invariance discussion in Section \ref{sec:gen.case} and standard
results from the invariance literature; see Section 6.6 of
\citealt{berger2013statistical}.)

Why do these priors fail to control the confidence level for
$n \geq 2$? We offer a heuristic argument. Recall that in the $n = 1$
case, the posterior for $\nu$ equals the prior,
$\pi(\nu|Y) = \pi(\nu)$, so a single observation provides no
information about $\nu$. This is what frees us to choose the prior
over $\nu$ to maximize a criterion of our choosing (fat tails), but
the resulting prior is highly informative (a point mass). When
$n \geq 2$, the data do inform $\nu$, so $\pi(\nu|Y) \neq \pi(\nu)$,
and this information can contradict the point-mass prior we
imposed. This is consistent with Figure~\ref{fig:cred.t1e.n2}, where
coverage is correct only at $\nu = 1$ (the point mass value for the
normal). It is possible that some other prior over $\nu$ would yield
valid confidence intervals for $n \geq 2$, but we have not found
one. Part of the difficulty is that the fattest-tails choice is
possible only because the data are silent on $\nu$ at $n = 1$. Once
the data speak, there is no longer an obvious criterion singling out a
prior over $\nu$.

\subsection{Augmented $t$-intervals}
\label{sec:aug.t.int.gen}

To generalize the methods of Section~\ref{sec:bf.invert} to $n \geq 2$
we note, as there, that the Bayes factor from
Theorem~\ref{theo:aug.t.as.bf} is a monotonic function of the
augmented $t$-statistic (Theorem~\ref{thm:bf.monotone}), where we have
data $X_1,\ldots,X_n$ and $A$ (see Appendix~\ref{sec:more.m} for
exploring using more augmented data). Thus, any confidence interval
that is inverted from a test using the Bayes factor is equivalent to
an augmented $t$-interval. Let $\hat{\mu}$ and $\hat{\sigma}^2$ be the
sample mean and sample variance using the augmented data. Then
inverting an augmented $t$-test results in intervals of the form
\begin{align}
  \label{eq:aug.t.int}
  \hat{\mu} \pm \eta \hat{\sigma}/\sqrt{n + 1},
\end{align}
where $\eta$ is chosen to control the error probability. Note that for
$n=1$, $\hat{\mu} = (X+A)/2$ and $\hat{\sigma} = |X-A|/\sqrt{2}$,
making this equivalent to the $n=1$ confidence intervals of
Appendix~\ref{sec:bm.deriv}. For large $n$, we also have that
$\hat{\mu} \approx \bar{X}$ and $\hat{\sigma} \approx S$, making it
asymptotically equivalent to the Student $t$-intervals. This approach
thus nicely bridges the standard Student $t$-intervals with the
classical $n=1$ confidence intervals.

To find $\eta$, we use the following form of the squared augmented
$t$-statistic derived in Theorem~\ref{theo:form.aug.t.n.2.orig}.
\begin{theorem}
  \label{theo:form.aug.t.n.2.orig}
  \begin{gather}
    \frac{(\hat{\mu} - \mu)^2}{\hat{\sigma}^2/(n + 1)} = \frac{(nZ - \sqrt{n}\nu)^2}{(n+1)W^2 + (Z + \sqrt{n}\nu)^2}, \text{ where}\\
    Z = \frac{\bar{X} - \mu}{\sigma/\sqrt{n}} \sim \N(0,1),~  W^2 =  \frac{(n-1)S^2}{\sigma^2} \sim \chi^2_{n-1}, \text{ and } \nu = \frac{\mu - A}{\sigma},
  \end{gather}
  and $Z$ and $W^2$ are independent.
\end{theorem}
\begin{proof}
  This is a special case of Theorem~\ref{theo:form.aug.t.n.2} in
  Appendix~\ref{sec:more.m}, where we set $m=1$.
\end{proof}

From Theorem~\ref{theo:form.aug.t.n.2.orig}, for a given $\nu$ and $\eta$, the
coverage failure probability is
\begin{align}
  \label{eq:alpha.n2.failure.def}
  \alpha(\nu,\eta) := \Pr_{\nu}\left[ \frac{\left(nZ - \sqrt{n}\nu\right)^2}{(n+1)W^2 + (Z + \sqrt{n}\nu)^2} > \eta^2\right].
\end{align}
Thus, to find $\eta$, we maximize \eqref{eq:alpha.n2.failure.def} over
$\nu$, which results in a worst-case $\alpha$ as a function of
$\eta$. We can then invert this function to obtain the value of $\eta$
given a worst-case $\alpha$, say $\eta_{\alpha}$. Since
\eqref{eq:alpha.n2.failure.def} is a double integral, it is possible
to calculate numerically. Specifically, let
\begin{align}
  g(Z,\nu,\eta) = \frac{(nZ - \sqrt{n}\nu)^2}{\eta^2(n+1)} - \frac{(Z + \sqrt{n}\nu)^2}{n+1},
\end{align}
then
\begin{align}
  \alpha(\nu,\eta) = \int_{-\infty}^{\infty}\int_{0}^{g(z,\nu,\eta)}\chi^2_{n-1}(w^2)\mathrm{d}w^2\phi(z)\mathrm{d}z,
\end{align}
where $\phi(\cdot)$ is the standard normal density and
$\chi^2_{n-1}(\cdot)$ is the $\chi^2_{n-1}$ density.

Table~\ref{tab:mult} contains, for $n = 2,\ldots,10,30$, the
multipliers to the sample standard deviations of either the original
or augmented data to obtain the $95\%$ Student or augmented
$t$-intervals, respectively. The multipliers for the augmented
$t$-intervals are smaller than for the Student $t$-intervals, and for
$n = 2$ or $3$ this reduction is large. If $A$ is near $\mu$, then
this would create augmented $t$-intervals which are much smaller than
the Student $t$-intervals. However, if $A$ is far from $\mu$, then
this would inflate the standard deviation of the augmented data and
the augmented $t$-intervals could be much larger than the Student
$t$-intervals. By $n = 30$, the multipliers are within 2 decimal
places of each other. This is not a coincidence, as they are
asymptotically equivalent, as shown by Theorem~\ref{theo:teta}.

\begin{theorem}
  \label{theo:teta}
  Let $\eta_{\alpha,n}$ be the multiplier of the $(1 - \alpha)100\%$
  augmented $t$-interval for a sample of size $n$, and let
  $t_{1-\alpha/2,n-1}$ be the $1-\alpha/2$ quantile of a
  $t$-distribution with $n -1$ degrees of freedom. Then
  $\eta_{\alpha,n} = t_{1-\alpha/2,n-1} +
  \mathcal{O}\left(\frac{1}{\sqrt{n}}\right)$.
\end{theorem}

We can compare the expected squared half-widths of the augmented and
Student $t$-intervals to see which are narrower on average under
different conditions. For the Student $t$-intervals, the expected
squared half-width is
\begin{align}
  \label{eq:e.hw.studt}
  \E\left[t_{1-\alpha/2,n-1}^2S^2/n\right] = \frac{t_{1-\alpha/2,n-1}^2\sigma^2}{n(n-1)}\E\left[(n-1)S^2 / \sigma^2\right] = \frac{1}{n}t_{1-\alpha/2,n-1}^2\sigma^2.
\end{align}
For the augmented $t$-interval, the expected squared half-width is
calculated in Theorem~\ref{theo:squared.half.width}

\begin{theorem}
  \label{theo:squared.half.width}
  The expected squared half-width of interval \eqref{eq:aug.t.int} is
  \begin{align}
    \label{eq:e.hw.augt}
    E[\eta^2 \hat{\sigma}^2 / (n+1)] = \frac{n + \nu^2}{(n+1)^2}\eta_{\alpha}^2\sigma^2,
  \end{align}
  where $\nu = (\mu - A) / \sigma$.
\end{theorem}

Comparing \eqref{eq:e.hw.studt} and \eqref{eq:e.hw.augt}, only the
value of $\nu$ determines whether the augmented or Student
$t$-interval has the shorter expected squared width. We plot the ratio
of the expected squared half-widths of the Student and augmented
$t$-intervals for various levels of $\nu$ in
Figure~\ref{fig:half.width}.  We see there that there are large
regions of the parameter space where the augmented $t$-interval has
shorter widths on average. For example, for $n=2$, $A$ just needs to be
within about 4.5 standard deviations of $\mu$ to have shorter length
than the Student $t$-interval. Of course, since the augmented and
Student $t$-intervals are asymptotically equivalent, this improvement
decreases with sample size, and for $n \geq 10$ the improvement is
already minimal. However, the region where the augmented $t$-interval
is better appears to be $|\mu - A| / \sigma < \sqrt{2}$ for $n$ large.

\section{Mean hyperbolic excess velocity of interstellar objects}
\label{sec:iso}

As of 5 June 2026, humanity has confirmed three interstellar objects
(ISOs) passing through the Solar System: 1I/'Oumuamua was discovered
on 19 October 2017 \citep{williams2017, meech2017brief}, 2I/Borisov on
30 August 2019 \citep{borisov2019, guzik2020initial}, and 3I/ATLAS on
1 July 2025 \citep{atlas2025, seligman2025discovery}. A key property
of each ISO is its hyperbolic excess velocity, $v_{\infty}$, the
object's speed relative to the Sun far from the Solar System, where
the Sun's gravity is negligible. It (together with the direction of
approach) records the object's motion through the local Galaxy before
the Solar System perturbed it. Under the assumption that ISOs inherit
the kinematics of their parent stellar populations, the distribution
of $v_{\infty}$ across ISOs is a probe of the age and Galactic origin
of the populations that produced them
\citep{hopkins2025predicting}. As of 5 June 2026, the $v_{\infty}$
for each object, taken from the Small-body Database of the Jet
Propulsion Laboratory (JPL), has been measured as:
\begin{itemize}[noitemsep]
\item 1I/'Oumuamua: 26.4 km/s \citep{sbdb_oumuamua},
\item 2I/Borisov: 32.3 km/s \citep{sbdb_borisov}, and
\item 3I/ATLAS: 58.0 km/s \citep{sbdb_atlas}.
\end{itemize}

Eight months prior to the discovery of 'Oumuamua,
\citet{engelhardt2017observational} predicted the distribution of
$v_{\infty}$ for ISOs to be about the same as the distribution of the
relative speed for nearby stars, which has mean 25 km/s and standard
deviation of 5 km/s. In this section, we use this prior knowledge to
construct our augmented $t$-intervals
(Section~\ref{sec:aug.t.int.gen}), and compare these to the FAB
intervals \citep{yu2018adaptive} and the standard Student
$t$-intervals \citep{student1908probable}.

All three interval methods assume $X_1,X_2,X_3 \overset{\text{iid}}{\sim} \N(\mu, \sigma^2)$
and produce valid $(1 - \alpha)100\%$ confidence intervals of
$\mu$. The augmented $t$-intervals require an augmented value, $A$,
which we choose to be 25 based on
\citet{engelhardt2017observational}. The FAB intervals require a prior
distribution over $\mu$, which we choose to be $\N(25, 25)$, again
based on \citet{engelhardt2017observational}. 95\% confidence
intervals calculated using either just 'Oumuamua ($n = 1$), just
'Oumuamua and Borisov ($n = 2$), or all three ISOs ($n = 3$) are
provided in Table~\ref{tab:ci.iso}. Only our augmented $t$-intervals
produce valid confidence intervals at $n = 1$. For $n \geq 2$, notice
that the Student $t$-intervals are much wider than the FAB and
augmented $t$-intervals. For $n = 2$, the augmented $t$-interval is
smaller than the FAB interval, and for $n = 3$ the FAB and augmented
$t$-intervals are about the same width.

It is of interest to see how a Bayesian analysis would proceed using
just the 'Oumuamua observation and the methods of
Section~\ref{sec:normal.family}. We plot the posterior density of
$\mu$ and the posterior predictive distribution of a new observation
in Figure~\ref{fig:post.iso.anal}. We see there that the new
observations, particularly the ATLAS one, are in the tails. The
two-tailed Bayesian $p$-value \citep{rubin1984bayesianly,
  gelman1996posterior} for the ATLAS observation given just the
'Oumuamua observation is 0.02464. This could either indicate that
ATLAS comes from a different stellar population (e.g., thin vs thick
disk) \citep{taylor2025kinematic} or indicate issues in our modeling
assumptions and that the distribution of $v_{\infty}$ is non-normal
\citep{hopkins2025predicting}.

\section{Discussion}
\label{sec:discussion}

In this paper, we developed two Bayesian formulations for the $n=1$
frequentist confidence intervals. The first derives priors whose
marginal posterior credible intervals are asymptotically (in the
confidence level) valid confidence intervals. The second inverts a
frequentist test that uses a Bayes factor (with appropriate priors) as
a test statistic. Our credible interval approach does not produce
valid confidence intervals for $n\geq 2$. But our Bayes factor
approach does, and these intervals have lower expected squared width
than the Student $t$-intervals for parts of the parameter space ---
and the improvements can be large for small $n$.

What other $n=1$ confidence intervals are there? In
\citet{blachman1987confidence}, they considered two others (and to our
knowledge, no other intervals have been thoroughly explored in the
literature):
\begin{align}
  \label{eq:ci.1.center.x}
  X &\pm \eta|X - A|, \text{ and}\\
  \label{eq:ci.1.center.a}
  A &\pm \eta|X - A|.
\end{align}
Our credible interval approach in Section~\ref{sec:gen.case} also
turns out to approximate intervals \eqref{eq:ci.1.center.x} and
\eqref{eq:ci.1.center.a}. This is because the tail probability of
these intervals is the same as that of \eqref{eq:ci.init.2}
\citep{blachman1987confidence}. We did not provide detailed
descriptions of these other intervals because the approximation is a
little worse. E.g., if $X \sim \Cauchy(\mu, \sigma)$, we observe
$x = 1$, and $A = 0$, then the 95\% CI of \eqref{eq:ci.1.center.x} is
(-5.4, 7.4), while the Bayesian 95\% credible interval is (-5.9,
6.9). If we instead look at the 99.9\% confidence and credible
intervals, we get (-317.3, 319.3) and (-317.8, 318.8), respectively,
with the approximation improving even more for larger confidence
levels. This is in contrast to interval \eqref{eq:ci.init.2}, where
the confidence and credible intervals exactly lined up in the Cauchy
case (Section~\ref{sec:cauchy.family}), and nearly so in the normal
case (Section~\ref{sec:normal.family}). \st{Waving our hands}
Heuristically, we suspect that this is because (at least in the normal
case) interval \eqref{eq:ci.init.2} is nearly optimal (in the sense of
\citealt{portnoy2019invariance}, based on the numerical results of
Section~\ref{sec:normal.family}), and Bayesian procedures tend to
produce frequentist procedures with good properties.

There is another strategy we could have taken in
Section~\ref{sec:gen.case} to derive priors for $\mu$ and
$\sigma^2$. We could have taken the confidence distribution implied by
\eqref{eq:ci.init.2} and just treated it as the marginal posterior
distribution of $\mu$. Since the prior is proportional to the
posterior divided by the likelihood, this strategy could potentially
yield priors that produce valid confidence intervals. However, there
are a couple of issues with this approach. First, the confidence
distribution implied by \eqref{eq:ci.init.2} does not exist between
the 25th and the 75th percentiles, so we would have to ``cleverly
impute'' this part of the confidence distribution. Another issue is
that this approach would not specify the posterior of $\sigma^2$ or
the posterior dependence between $\mu$ and $\sigma^2$, which would
need to be derived separately. Our approach from
Section~\ref{sec:gen.case}, on the other hand, uses simple priors to
produce credible intervals that are (asymptotically in the confidence
level) equivalent to the confidence distribution implied by
\eqref{eq:ci.init.2}.

The prior we developed in Section~\ref{sec:gen.case} can be considered
a type of ``probability matching prior'' (PMP)
\citep{lindley1958fiducial, welch1963formulae, peers1965confidence,
  mukerjee1997second, severini2002exact, datta2005probability,
  staicu2008probability, diciccio2017simple}. A PMP is one for which
posterior quantiles yield valid frequentist confidence intervals. Most
PMPs are approximate in the sample size
\citep{datta2005probability}. There are exact results for
location-scale models \citep{severini2002exact, diciccio2017simple};
however, these results use the prior $\pi(\mu,\sigma) = \sigma^{-1}$,
which (in the normal case) yields the standard $t$-intervals, produces
non-integrable posteriors when $n = 1$ (since the posterior is
proportional to $1 / |\mu-x|$), and has been known in general
location-scale models since at least \citet{peers1965confidence}. The
PMP we derived in Section~\ref{sec:gen.case} is new and produces the
$n = 1$ confidence intervals (asymptotically in the confidence level).

In Section~\ref{sec:bf.invert} we used an improper prior for the
variance in calculating the Bayes factor. Using an improper prior in
Bayes factors can be problematic because the final value of the Bayes
factor is only defined up to some undefined multiplicative constant
\citep{kass1995bayes}. However, some justification is provided if the
improper prior is placed on parameters that are shared in both
hypotheses (as in our case) because the indeterminate multiplicative
constant will cancel in the Bayes factor \citep{jeffreys1998theory,
  robert1993note, berger2001objective}. See also
\citet{sanso1996robustness} for more formal justification. However,
particularly in our case, where we are using the Bayes factor as a
frequentist test statistic, it is completely justified to use improper
priors for \emph{any} parameter (not just nuisance parameters), as the
undefined multiplicative constant, once one is arbitrarily chosen,
will just be incorporated in the null distribution one uses to
calibrate the Bayes factor.

\subsection*{Acknowledgments}
\label{sec:ack}

Many thanks to Professor Peter Hoff, Duke University, for providing
comments on a draft of this manuscript.

Most analyses were performed using the R statistical language
\citep{rcoreteam}.

\bibliography{geno_bib.bib}

\subsection*{Data accessibility}
\label{sec:data.available}

All of the methods described in this manuscript are implemented in the
\texttt{nisone} R package on GitHub:\\
\url{https://github.com/dcgerard/nisone}

All analyses in this manuscript are completely reproducible with
executable code, on GitHub:\\
\url{https://github.com/dcgerard/reproduce_nisone}


\clearpage

\section{Figures and Tables}
\label{sec:figs}

\begin{figure}[!htb]
  \centering
  \includegraphics{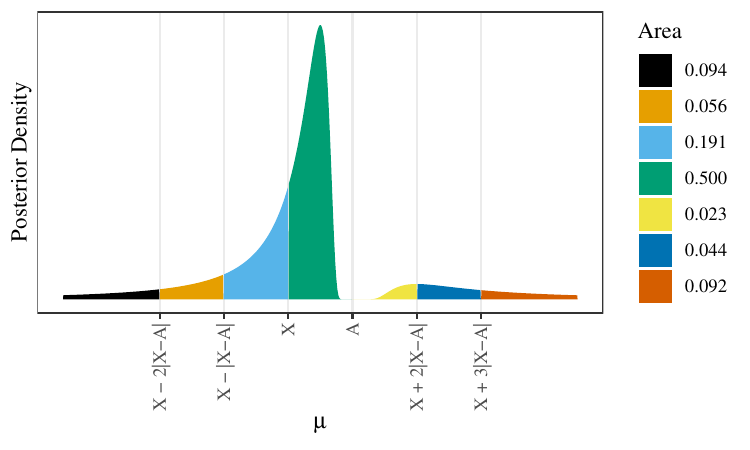}
  \caption{Posterior density of $\mu$ given $X$ when
    $X\sim N(\mu, \sigma^2)$, using prior \eqref{eq:prior.best}, and
    $A > X$. Modes are at $(X + A)/2$ and $A + |X - A|$. The areas
    between values $|X-A|$ units apart are color coded.}
  \label{fig:post.dens.graphic}
\end{figure}

\begin{table}[!htb]
\centering
\begin{tabular}{rrrrrr}
  \hline
  $\alpha$ & lf & lb & uf & ub & Worst $\alpha$ \\
  \hline
  0.30 & -0.94 & -0.72 & 1.94 & 2.00 & 0.3168288 \\
  0.20 & -1.81 & -1.77 & 2.81 & 2.84 & 0.2009078 \\
  0.10 & -4.29 & -4.28 & 5.29 & 5.29 & 0.1000218 \\
  0.05 & -9.15 & -9.15 & 10.15 & 10.15 & 0.0500006 \\
  0.01 & -47.89 & -47.89 & 48.89 & 48.89 & 0.0100000 \\
  \hline
\end{tabular}
\caption{Lower and upper bounds for frequentist (lf and uf) and new
  Bayesian (lb and ub) approaches for different values of $\alpha$
  when $X = 1$ and $A = 0$. The ``Worst $\alpha$'' column contains the
  worst-case error probability of the Bayesian credible interval. The
  confidence level approaches the nominal level for
  $\alpha \rightarrow 0$.}
\label{tab:compare.bayes}
\end{table}

\begin{table}[!htb]
\centering
\begin{tabular}{rrr}
  \hline
  Level & Augmented $t$ & BF \\
  \hline
  0.80 & 4.62 & 1.73 \\
  0.90 & 9.57 & 2.49 \\
  0.95 & 19.31 & 3.27 \\
  0.99 & 96.78 & 5.02 \\
  \hline
\end{tabular}
\caption{The level of the $n=1$ confidence intervals, along with
  maximum augmented $t$-statistic ($c_1$ in \eqref{eq:standard.t.int})
  and maximum BF ($c_2$ in \eqref{eq:bf.equivalent.int}).}
\label{tab:bf.t}
\end{table}

\begin{figure}[!htb]
  \centering
  \includegraphics{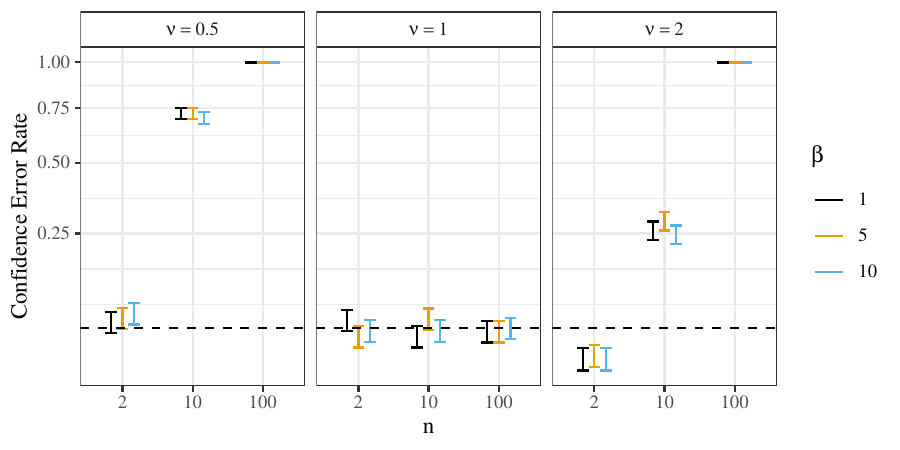}
  \caption{Confidence error rate ($y$-axis, square-root scale) for
    95\% credible intervals based on posterior
    \eqref{eq:kernel.inv.norm.n2} for $n = 2, 10, 100$ ($x$-axis),
    $\beta = 1, 5, 10$ (color), and $\nu = 0.5, 1, 2$ (facets). Exact
    binomial confidence intervals are plotted based on the 1000
    replications for each scenario. The error rate should be at or
    below 0.05 (the horizontal dashed line) to be valid 95\%
    confidence intervals. Only when $\nu = 1$ do we get valid
    confidence intervals across conditions.}
  \label{fig:cred.t1e.n2}
\end{figure}

\begin{table}[!htb]
  \centering
  \begin{tabular}{rrrr}
    \hline
    $n$ & $\eta_{0.05}$ & $\eta/\sqrt{n+1}$ & $t_{0.975,n-1}/\sqrt{n}$ \\
    \hline
    2 & 5.79 & 3.35 & 8.98 \\
    3 & 3.75 & 1.88 & 2.48 \\
    4 & 3.04 & 1.36 & 1.59 \\
    5 & 2.71 & 1.11 & 1.24 \\
    6 & 2.53 & 0.96 & 1.05 \\
    7 & 2.42 & 0.86 & 0.92 \\
    8 & 2.35 & 0.78 & 0.84 \\
    9 & 2.29 & 0.72 & 0.77 \\
    10 & 2.25 & 0.68 & 0.72 \\
    30 & 2.04 & 0.37 & 0.37 \\
    \hline
  \end{tabular}
  \caption{The multiplier $\eta$ for interval \eqref{eq:aug.t.int} for
    $n=2,\ldots,10,30$ when $\alpha = 0.05$. This table also has
    $\eta/\sqrt{n+1}$, which is multiplied by the sample standard
    deviation of $X_1,\ldots,X_n,A$ to get the half-width of the
    augmented $t$-interval. For comparison, this table also presents
    the 0.975 quantile of the $t$-distribution with $n-1$ degrees of
    freedom divided by $\sqrt{n}$, which is multiplied by the sample
    standard deviation of the original data to get the half-width of
    the Student $t$-interval.}
  \label{tab:mult}
\end{table}

\begin{figure}[!htb]
  \centering
  \includegraphics{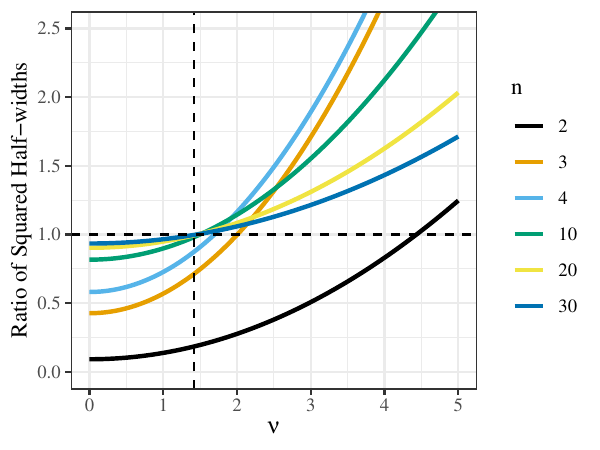}
  \caption{Ratio of 95\% augmented $t$ to Student $t$ CI expected
    squared half-widths ($y$-axis) for different values of the
    standardized deviation from the mean ($x$-axis),
    $\nu = (\mu - A) / \sigma$. Values below 1 (dashed horizontal
    line) indicate that the augmented $t$-interval has shorter
    expected squared width. The vertical dashed line is at $\sqrt{2}$,
    the value of $\nu$ above which the Student $t$-interval appears to
    asymptotically have shorter squared expected length than the
    augmented $t$-interval.}
  \label{fig:half.width}
\end{figure}

\begin{table}[!htb]
  \centering
  \begin{tabular}{rlll}
    \hline
    $n$ & Augmented $t$ & FAB & Student $t$ \\
    \hline
    1 & (12.2, 39.2) & NA & NA \\
    2 & (14.9, 40.9) & (9.9, 48.1) & (-8.1, 66.8) \\
    3 & (6.6, 64.3) & (9.9, 67.2) & (-2.8, 80.6) \\
    \hline
  \end{tabular}
  \caption{95\% confidence intervals for the mean $v_{\infty}$ for
    ISOs based on just 'Oumuamua ($n = 1$), on just 'Oumuamua and
    Borisov ($n = 2$), or on all three ISOs ($n = 3$). Intervals
    constructed are the augmented $t$-intervals
    (Section~\ref{sec:aug.t.int.gen}), FAB intervals
    \citep{yu2018adaptive}, and Student $t$-intervals
    \citep{student1908probable}.}
  \label{tab:ci.iso}
\end{table}

\begin{figure}[!htb]
  \centering
  \includegraphics{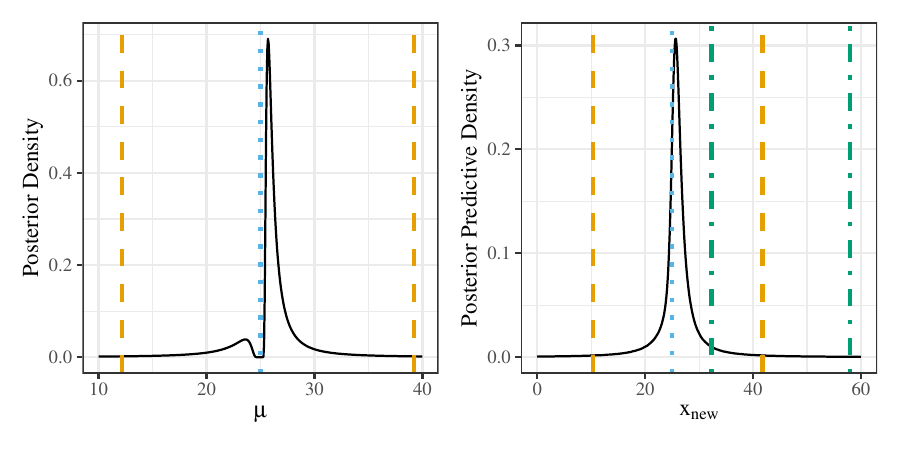}
  \caption{Posterior density of $\mu$ (left facet) and posterior
    predictive distribution for a new observation (right facet) based
    only on the 'Oumuamua observation. The 0.025 and 0.975 quantiles
    are indicated by the orange dashed lines, the value $A$ used is
    indicated by the blue dotted line, and the Borisov and ATLAS
    observations are indicated by the green dot-dash lines.}
  \label{fig:post.iso.anal}
\end{figure}

\clearpage
\appendix

\setcounter{table}{0}
\renewcommand{\thetable}{S\arabic{table}}
\renewcommand{\theHtable}{S\arabic{table}}
\setcounter{figure}{0}
\renewcommand{\thefigure}{S\arabic{figure}}
\renewcommand{\theHfigure}{S\arabic{figure}}
\setcounter{equation}{0}
\renewcommand{\theequation}{S\arabic{equation}}
\renewcommand{\theHequation}{S\arabic{equation}}
\setcounter{section}{0}
\renewcommand{\thesection}{S\arabic{section}}
\renewcommand{\theHsection}{S\arabic{section}}
\setcounter{lemma}{0}
\renewcommand{\thelemma}{S\arabic{lemma}}
\renewcommand{\theHlemma}{S\arabic{lemma}}
\setcounter{theorem}{0}
\renewcommand{\thetheorem}{S\arabic{theorem}}
\renewcommand{\theHtheorem}{S\arabic{theorem}}
\setcounter{corollary}{0}
\renewcommand{\thecorollary}{S\arabic{corollary}}
\renewcommand{\theHcorollary}{S\arabic{corollary}}
\setcounter{algorithm}{0}
\renewcommand{\thealgorithm}{S\arabic{algorithm}}
\renewcommand{\theHalgorithm}{S\arabic{algorithm}}

\section{$n=1$ confidence intervals}
\label{sec:bm.deriv}

Here, we review the $n=1$ confidence intervals from
\citet{blachman1987confidence} and derive the tail approximations for
its confidence distribution. We provide these details for
completeness, and because \citet{blachman1987confidence} do not derive
the order of the approximation error.

Suppose that we have a location-scale family
\begin{align}
f(X|\mu,\sigma) &= \frac{1}{\sigma}\rho\left(\frac{X-\mu}{\sigma}\right),
\end{align}
and we consider intervals of the form
\begin{align}
  (X + A)/2 \pm \eta |X - A|,
\end{align}
where $A$ is some pre-specified value. The failure region of this
interval is $|\mu - (X+A)/2| > \eta |X-A|$. Let
$Z = (X - \mu) / \sigma$ and $\nu = (\mu - A) / \sigma$.  Then this
failure region is equivalent to $|Z - \nu|/2 > \eta|Z + \nu|$. This
inequality is satisfied if (Figure~\ref{fig:bm.bounds})
\begin{align}
  \label{eq:z.bm.bounds}
  Z \in \left(-\nu\frac{2\eta - 1}{2\eta + 1}, -\nu\frac{2\eta + 1}{2\eta - 1}\right).
\end{align}
Since $Z$ has density $\rho(z)$, we can calculate this failure
probability
\begin{align}
  \label{eq:alpha.bm}
  \alpha &= \left| \int_{-\nu(2\eta - 1) / (2\eta + 1)}^{-\nu(2\eta + 1) / (2\eta - 1)} \rho(z) \mathrm{d}z \right|\\
         &= \left| \int_{\nu(2\eta - 1) / (2\eta + 1)}^{\nu(2\eta + 1) / (2\eta - 1)} \rho(z) \mathrm{d}z \right|,
\end{align}
where the second equality follows since $\rho(\cdot)$ is symmetric.

\begin{figure}[!htb]
  \centering
  \includegraphics{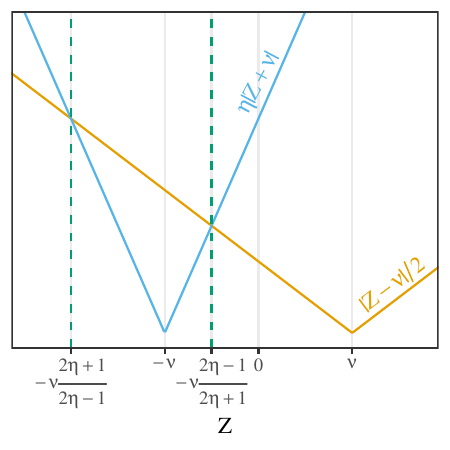}
  \caption{Visualization of bounds for \eqref{eq:z.bm.bounds}. The
    failure region is when $Z$ is between the two dashed green lines.}
  \label{fig:bm.bounds}
\end{figure}

An exact approach not relying on tail approximations finds the
maximum value of $\alpha$ \eqref{eq:alpha.bm} for a fixed value of
$\eta$ (maximizing over $\nu$). This results in the worst-case
$\alpha$ as some function of $\eta$, which can be inverted to find the
$\eta$ given a worst-case $\alpha$. In the normal case, we invert the
following function of $\eta$,
\begin{align}
  \label{eq:alpha.x2cen}
  \alpha(\eta) &= \Phi\left(\tau(\eta)\frac{2\eta + 1}{2\eta-1}\right) - \Phi\left(\tau(\eta)\frac{2\eta - 1}{2\eta+1}\right), \text{ where}\\
  \label{eq:tau.def}
  \tau(\eta) &= (4\eta^2 - 1)\sqrt{\frac{\arcoth(2\eta)}{2(4\eta^3 + \eta)}}, \text{ and } \arcoth(\eta) = \frac{1}{2}\log\left(\frac{\eta+1}{\eta-1}\right).
\end{align}
This has a solution only for $\alpha \leq 1/2$ (or $\eta \geq
1/2$). Note that \eqref{eq:tau.def} is a corrected version of the
equation at the bottom of p.~376 in \citet{blachman1987confidence},
which is missing a $\sqrt{2}$ term.

For an approximate approach, we can do a Taylor series expansion of
\eqref{eq:alpha.bm} at $\eta = \infty$. Let $\zeta = 1 / \eta$, so we
will do a Taylor series expansion at $\zeta = 0$. For $\nu > 0$, we
have
\begin{align}
  \alpha &= \int_{\nu(2\eta - 1) / (2\eta + 1)}^{\nu(2\eta + 1) / (2\eta - 1)} \rho(z) \mathrm{d}z\\
  &= \int_{\nu(2 - \zeta) / (2 + \zeta)}^{\nu(2 + \zeta) / (2 - \zeta)} \rho(z) \mathrm{d}z\\
  \begin{split}
    &= \int_{\nu(2 - 0) / (2 + 0)}^{\nu(2 + 0) / (2 - 0)} \rho(z) \mathrm{d}z + \bigg[\rho\left(\nu \frac{2 + 0}{2 - 0}\right)\nu\left(\frac{1}{2-0}+\frac{2 + 0}{(2-0)^2}\right) - \\
    &\phantom{= \int_{\nu(2 - 0) / (2 + 0)}^{\nu(2 + 0) / (2 - 0)} \rho(z) \mathrm{d}z + \bigg[}\rho\left(\nu\frac{2 - 0}{2 + 0}\right)\nu\left(\frac{-1}{2+0}-\frac{2 - 0}{(2+0)^2}\right)\bigg]\zeta + \mathcal{O}(\zeta^2)
  \end{split}\\
  &= 2\nu\rho(\nu)\zeta + \mathcal{O}(\zeta^2)\\
  &= \frac{2}{\eta}\nu\rho(\nu) + \mathcal{O}\left(\frac{1}{\eta^2}\right).
\end{align}
The largest value of $\alpha$, in the tails, is then approximately
$\frac{2}{\eta}\max_{\nu>0}\nu\rho(\nu)$. Since this is an equal
tailed confidence interval, we can divide this approximation by two to
get the tail approximation of the confidence distribution:
\begin{align}
  &\Pr(\mu \geq (X+A)/2 + \eta|X-A|) = \frac{1}{\eta}\max_{\nu>0}\nu\rho(\nu) + \mathcal{O}\left(\frac{1}{\eta^2}\right)\\
  &\Rightarrow \Pr(\mu \geq \eta|X-A|) = \frac{1}{\eta}\max_{\nu>0}\nu\rho(\nu) + \mathcal{O}\left(\frac{1}{\eta^2}\right) \text{ (center does not affect tail probabilities)}\\
  &\Rightarrow \Pr(\mu \geq z) = \frac{|X-A|}{z}\max_{\nu>0}\nu\rho(\nu) + \mathcal{O}\left(\frac{1}{z^2}\right) \text{ (setting $z = \eta|X-A|$)}.
\end{align}
Setting $Y = X - A$ and $\beta = \mu - A$, as we do in
Section~\ref{sec:gen.case}, we get the confidence distribution for
$\beta$,
\begin{align}
  \Pr(\beta \geq z) = \frac{|Y|}{z}\max_{\nu>0}\nu\rho(\nu) + \mathcal{O}\left(\frac{1}{z^2}\right).
\end{align}

\section{Proofs of results}
\label{sec:proofs}

\paragraph{Lemma~\ref{lem:y.dense}}
\begin{proof}
   We can write the likelihood of $X$ as
   \begin{align}
     f(x|\mu,\nu) &= \frac{\nu}{|\mu - A|}\rho\left(\nu\frac{x - \mu}{|\mu - A|}\right)\\
                  &= \frac{\nu}{|\mu - A|}\rho\left(\nu\frac{x - \mu}{\mu - A}\right) \text{ (symmetry)}\\
                  &= \frac{\nu}{|\mu - A|}\rho\left(\nu\frac{(x - A) - (\mu - A)}{\mu - A}\right)\\
                  &= \frac{\nu}{|\mu - A|}\rho\left(\nu\frac{x - A}{\mu - A} - \nu\right).
   \end{align}
   Setting $Y = X - A$ and $\beta = \mu - A$, and doing a change of variables, we get \eqref{eq:likelihood.beta.nu}.
 \end{proof}

 \paragraph{Theorem~\ref{theo:post.ni1}}
 \begin{proof}
  Since the kernel of the posterior is
  \begin{align}
    \pi(\beta,\nu|Y) &\propto \frac{\nu}{|\beta|}\rho\left(\nu\frac{Y}{\beta} - \nu\right)|\beta|^{-1}\pi(\nu)\\
                     &= \frac{\nu}{\beta^2}\rho\left(\nu\frac{Y}{\beta} - \nu\right)\pi(\nu)\\
                     &\propto \eqref{eq:ls.post.full},
  \end{align}
  it suffices to show that \eqref{eq:ls.post} is a density. Let
  $a = \frac{1}{\beta}$, then a change of variables results in
  \begin{align}
    \pi(a|\nu,Y) &= \nu |Y|\rho\left(a\nu Y-\nu\right)\\
                 &= \frac{1}{1/|\nu Y|}\rho\left(\frac{a-1/Y}{1/(\nu Y)}\right)\\
    \label{eq:post.form.beta}
                 &= \frac{1}{1/|\nu Y|}\rho\left(\frac{a-1/Y}{1/|\nu Y|}\right) \text{ (symmetry)}.
  \end{align}
  Inspecting \eqref{eq:post.form.beta}, we see that this density for
  $1/\beta$ is in same location-scale family as $Y$ but with location
  $\frac{1}{Y}$ and scale $\frac{1}{\nu |Y|}$.
\end{proof}

\paragraph{Theorem~\ref{theo:tail.prob}}
\begin{proof}
  We set $w = 1/z$ and $t = 1/\beta$, then we do a Taylor series of
  the tail probability at $w=0$ (the same as $z = \infty$) to get
  \eqref{eq:tail.prob}.
  \begin{align}
    \Pr(\beta > z|\nu,Y) &= \int_{z}^{\infty}\frac{\nu |Y|}{\beta^2}\rho\left(\nu\frac{Y}{\beta}-\nu\right)\mathrm{d}\beta\\
    \label{eq:int.a.given.1.over.z}
                         &= \int_0^{w}\nu |Y|\rho\left(t\nu Y-\nu\right)\mathrm{d}t\\
                         &= \int_0^{0}\nu |Y|\rho\left(t\nu Y - \nu\right)\mathrm{d}t + \nu |Y|\rho\left(-\nu\right)w + R(w,\nu)\\
                         &=\frac{|Y|}{z}\nu\rho\left(-\nu\right) + R(z,\nu)\\
    \label{eq:taylor.series.orig.proof}
                         &=\frac{|Y|}{z}\nu\rho\left(\nu\right) + R(z,\nu).
  \end{align}
  $R(z,\nu)$ is the remainder term, and by Taylor's Theorem $R(z, \nu) = \mathcal{O}\left(\frac{1}{z^2}\right)$, proving \eqref{eq:tail.prob}.
\end{proof}

\paragraph{Theorem~\ref{theo:valid.conf.post}}
\begin{proof}
  The proof is basically quantile matching. Equation
  \eqref{eq:post.tails} specifies the approximate tail area for any
  threshold $z$. We can also get the approximate tail area for any
  threshold $z$ from the confidence distribution of the interval
  $\frac{y}{2} \pm \eta|y|$ (Appendix~\ref{sec:bm.deriv}). These tail
  areas turn out to be the same up to an error term that is
  $\mathcal{O}(\alpha^2)$.

  Appendix~\ref{sec:bm.deriv} shows that the error probability of
  interval \eqref{eq:ci.init.2} is
  \begin{align}
    \alpha = \Pr(|\beta - Y/2| > \eta|Y|) = \frac{2}{\eta}\tilde{\nu}\rho(\tilde{\nu}) + \mathcal{O}\left(\frac{1}{\eta^2}\right).
  \end{align}
  Now treating $\beta$ as a random variable from the confidence
  distribution \citep{xie2013confidence, schweder2016confidence} of \eqref{eq:ci.init.2}, we have
  \begin{align}
    \alpha/2 &= \Pr(\beta - Y/2 > \eta|Y|) = \frac{1}{\eta}\tilde{\nu}\rho(\tilde{\nu}) + \mathcal{O}\left(\frac{1}{\eta^2}\right)\\
             &\Rightarrow \Pr(\beta > \eta|Y|) = \frac{1}{\eta}\tilde{\nu}\rho(\tilde{\nu}) + \mathcal{O}\left(\frac{1}{\eta^2}\right) \text{ (center does not affect tail probabilities)}\\
    \label{eq:conf.dist.tails}
             &\Rightarrow \Pr(\beta > z) = \frac{|Y|}{z}\tilde{\nu}\rho(\tilde{\nu}) + \mathcal{O}\left(\frac{1}{z^2}\right) \text{ (setting $z = \eta|Y|$)}.
  \end{align}
  Comparing the posterior tails \eqref{eq:post.tails} with the tails
  of the confidence distribution \eqref{eq:conf.dist.tails}, we get
  that the posterior tail probabilities are approximately the error
  probabilities of the confidence interval \eqref{eq:ci.init.2},
  \begin{align}
    \Pr(\beta \geq z|Y) = \alpha/2 + \mathcal{O}\left(\frac{1}{z^2}\right).
  \end{align}
  Note that $\alpha$ is of the same order as $\frac{1}{z}$, which leads us to
  \begin{align}
    \Pr(\beta \geq z|Y) = \alpha/2 + \mathcal{O}\left(\alpha^2\right).
  \end{align}
\end{proof}

\paragraph{Theorem~\ref{theo:confidence.dist.cauchy}}
\begin{proof}
   From Equation (13) of \citet{blachman1987confidence}, the confidence error for the interval $Y/2 \pm \eta |Y|$ for a given $\eta \geq 1/2$ and $\nu > 0$ is
   \begin{align}
     \label{eq:cauchy.a.nu.alpha}
    \alpha &= \left|\int_{\frac{\nu(2\eta - 1)}{2\eta + 1}}^{\frac{\nu(2\eta + 1)}{2\eta-1}} \rho(u)\mathrm{d}u\right|
    = \left|\int_{\nu/a}^{\nu a} \rho(u)\mathrm{d}u\right|,
  \end{align}
  for $a = \frac{2\eta + 1}{2\eta-1}$. Taking derivatives of
  \eqref{eq:cauchy.a.nu.alpha}, we get
  \begin{align}
    \label{eq:a.deriv.cauchy.nu}
    \frac{a}{\pi(1 + \nu^2a^2)} - \frac{1/a}{\pi(1 + \nu^2/a^2)}.
  \end{align}
  Setting \eqref{eq:a.deriv.cauchy.nu} equal to 0 and solving for
  $\nu$, we get $\nu = 1$. One can check this is a maximum. Thus, for
  any $\eta$, the confidence error rate is maximized at $\nu = 1$.

  We will use the following identities of $\arctan(\cdot)$, derived
  from Equations (4.4.16) and (4.4.34) of
  \citet{abramowitz1964handbook},
  \begin{align}
    \label{eq:arctan.neg}
    \arctan(-x) &= -\arctan(x),\\
    \label{eq:arctan.inv}
    \arctan(x) + \arctan\left(\frac{1}{x}\right) &= \frac{\pi}{2}\sign(x), \text{ and}\\
    \label{eq:arctan.rat}
    \arctan\left(\frac{x + 1}{x-1}\right) &= \frac{3\pi}{4} - \arctan(x) \text{ for } x \geq 1.
  \end{align}
  The confidence error \eqref{eq:cauchy.a.nu.alpha} for the interval
  $Y/2 \pm \eta |Y|$ for $\nu = 1$ is
  \begin{align}
    \alpha &= \left|\int_{\frac{2\eta - 1}{2\eta + 1}}^{\frac{2\eta + 1}{2\eta-1}} \frac{1}{\pi(1 + u^2)}\mathrm{d}u\right|\\
           &=\frac{1}{\pi}\arctan\left(\frac{2\eta + 1}{2\eta-1}\right) - \frac{1}{\pi}\arctan\left(\frac{2\eta - 1}{2\eta + 1}\right)\\
           &=\frac{2}{\pi}\arctan\left(\frac{2\eta + 1}{2\eta-1}\right) - \frac{1}{2}\\
           &= 1 - \frac{2}{\pi}\arctan(2\eta).
  \end{align}
  Because the confidence distribution is symmetric about $Y/2$, we
  have,
  \begin{align}
    &\Pr\left(\frac{Y}{2} + \eta |Y| \leq \beta\right) = \Pr\left(\frac{Y}{2} - \eta |Y| \geq \beta\right) = \frac{\alpha}{2} = \frac{1}{2} - \frac{1}{\pi}\arctan(2\eta)\\
    &\Leftrightarrow \Pr(Z \geq 2\eta) = \Pr(Z \leq -2\eta) = \frac{1}{2} - \frac{1}{\pi}\arctan(2\eta)\\
    \label{eq:cdf.cauchy}
    &\Leftrightarrow \Pr(Z \leq z) = \frac{1}{\pi}\arctan(z) + \frac{1}{2},
  \end{align}
  for $|z| \geq 1$. Equation \eqref{eq:cdf.cauchy} is exactly the CDF
  of the $\Cauchy(0, 1)$ distribution.
\end{proof}

\paragraph{Theorem~\ref{theo:posterior.dist.cauchy}}
\begin{proof}
  Using Theorem~\ref{theo:post.ni1}, the posterior distribution of
  $1/\beta$ is $\Cauchy(1 / Y,1 /|Y|)$ or, using the complex notation
  of \citet{mccullagh1992conditional},
  $1/\beta | Y \sim \Cauchy\left(\frac{1}{Y} + \frac{1}{Y}i\right)$. Thus,
  using Equation (2) of \citet{mccullagh1992conditional}, $\beta$ is
  Cauchy with parameter
  \begin{align}
    \frac{1}{\frac{1}{Y} + \frac{1}{Y}i} = \frac{Y}{1 + i} = Y\left(\frac{1}{2} - \frac{1}{2}i\right) = \frac{Y}{2} - \frac{Y}{2}i.
  \end{align}
  Going back to the location-scale parameterization of the Cauchy, we
  have $\beta|Y \sim \Cauchy\left(\frac{Y}{2}, \frac{|Y|}{2}\right)$,
  or $Z|Y \sim \Cauchy(0, 1)$.
\end{proof}

\paragraph{Theorem~\ref{thm:bf.monotone}}
\begin{proof}
Let
\begin{align}
  g(t) = \frac{\int_{\xi}T_{\nu}(t|\xi)\pi(\xi)\mathrm{d}\xi}{T_{\nu}(t)},
\end{align}
and let $\pi(\xi)$ be the density of $\xi$, which is assumed to be
symmetric about 0.  Then in this section, we will prove that $g(t)$ is
monotone increasing in $|t|$. This results in a proof of
Theorem~\ref{thm:bf.monotone} by setting $\nu = n$ and
$\xi = \sqrt{n + 1}\delta$, and noting that the density of $\delta$ is
also symmetric about 0.

Since $\pi(\xi)$ is symmetric about
0, and $T_\nu(-t|\xi) = T_\nu(t|-\xi)$, this implies that $g(\cdot)$ is
also symmetric about 0 since
\begin{align}
  g(-t) &= \frac{\int_{\xi}T_{\nu}(-t|\xi)\pi(\xi)\mathrm{d}\xi}{T_{\nu}(-t)}\\
                  &= \frac{\int_{\xi}T_{\nu}(t|-\xi)\pi(\xi)\mathrm{d}\xi}{T_{\nu}(t)}\\
                  &= \frac{\int_{\gamma}T_{\nu}(t|\gamma)\pi(-\gamma)\mathrm{d}\gamma}{T_{\nu}(t)} \text{ (change of variables, } \gamma = -\xi\text{)}\\
                  &= \frac{\int_{\gamma}T_{\nu}(t|\gamma)\pi(\gamma)\mathrm{d}\gamma}{T_{\nu}(t)} \text{ (symmetry)}\\
                  &= g(t).
\end{align}
It thus suffices to show that $g(t)$ is monotone increasing for $t > 0$.

We can re-write $g(\cdot)$ as
\begin{align}
  g(t) &= \frac{\int_{\xi}T_{\nu}(t|\xi)\pi(\xi)\mathrm{d}\xi}{T_{\nu}(t)}\\
                 &=\frac{\int_{\xi=0}^{\infty}T_{\nu}(t|\xi)\pi(\xi)\mathrm{d}\xi + \int_{\xi=-\infty}^{0}T_{\nu}(t|\xi)\pi(\xi)\mathrm{d}\xi}{T_{\nu}(t)}\\
                 &=\frac{\int_{\xi=0}^{\infty}T_{\nu}(t|\xi)\pi(\xi)\mathrm{d}\xi + \int_{\xi=0}^{\infty}T_{\nu}(t|-\xi)\pi(\xi)\mathrm{d}\xi}{T_{\nu}(t)} \text{ (change of variables + symmetry)}\\
                 &=\int_{\xi=0}^{\infty}\frac{T_{\nu}(t|\xi) + T_{\nu}(-t|\xi)}{{T_{\nu}(t)}}\pi(\xi)\mathrm{d}\xi.
\end{align}
It thus suffices to prove that $[T_{\nu}(t|\xi) + T_{\nu}(-t|\xi)]/T_{\nu}(t)$ is monotone increasing in $t > 0$ for any $\xi > 0$.

  From \citet{kruskal1954monotonicity}, we have the following representation of $T_{\nu}(t|\xi)$,
  \begin{align}
    T_{\nu}(t|\xi) &= \frac{\Gamma(\nu+1)}{2^{\frac{1}{2}(\nu-1)}\Gamma(\frac{\nu}{2})\sqrt{\pi \nu}}\left(\frac{\nu}{\nu+t^2}\right)^{\frac{1}{2}(\nu+1)}\exp\left\{-\frac{1}{2}\frac{\nu\xi^2}{\nu + t^2}\right\} Hh_{\nu}\left(\frac{-\xi t}{\sqrt{\nu + t^2}}\right), \text{ where}\\
    Hh_{\nu}(x) &= \int_0^{\infty}\frac{z^{\nu}}{\Gamma(\nu+1)}\exp\left\{-\frac{1}{2}(z + x)^2\right\}\mathrm{d}z.
  \end{align}
  We calculate
  \begin{align}
    \label{eq:log.t.diff.ratio}
    \log\left[\frac{T_{\nu}(t|\xi) + T_{\nu}(-t|\xi)}{{T_{\nu}(t)}}\right] = -\frac{1}{2}\frac{\nu\xi^2}{\nu + t^2} + \log\left[Hh_{\nu}\left(\frac{-\xi t}{\sqrt{\nu + t^2}}\right) + Hh_{\nu}\left(\frac{\xi t}{\sqrt{\nu + t^2}}\right)\right] - \log\left[Hh_{\nu}(0)\right].
  \end{align}
  We now replace $t$ with the strictly increasing function of it
  \begin{align}
    u = \frac{\xi t}{\sqrt{\nu + t^2}}, \text{  } t = u \sqrt{\frac{\nu}{\xi^2 - u^2}}.
  \end{align}
  Then it suffices to show that the following function of $u$ is
  monotone increasing for $u > 0$,
  \begin{align}
    \label{eq:fun.of.u}
    -\frac{1}{2}(\xi^2 - u^2) + \log\left[Hh_{\nu}(-u) + Hh_{\nu}(u)\right] - \log\left[Hh_{\nu}(0)\right].
  \end{align}
  Taking the derivative of \eqref{eq:fun.of.u} with respect to $u$, we get,
  \begin{align}
    \label{eq:deriv.of.fun.of.u}
    u + \frac{\int_{0}^{\infty}z^{\nu}(z - u) \exp\left\{-\frac{1}{2}(z-u)^2\right\}\mathrm{d}z - \int_{0}^{\infty}z^{\nu}(z + u) \exp\left\{-\frac{1}{2}(z+u)^2\right\}\mathrm{d}z}{\int_{0}^{\infty}z^{\nu}\exp\left\{-\frac{1}{2}(z-u)^2\right\}\mathrm{d}z + \int_{0}^{\infty}z^{\nu}\exp\left\{-\frac{1}{2}(z+u)^2\right\}\mathrm{d}z}.
  \end{align}
  We need to show that \eqref{eq:deriv.of.fun.of.u} is always greater than 0. Since the denominator of the fraction in \eqref{eq:deriv.of.fun.of.u} is greater than 0, \eqref{eq:deriv.of.fun.of.u} is greater than 0 if and only if
  \begin{align}
    \begin{split}
      0 &< u\left[\int_{0}^{\infty}z^{\nu}\exp\left\{-\frac{1}{2}(z-u)^2\right\}\mathrm{d}z + \int_{0}^{\infty}z^{\nu}\exp\left\{-\frac{1}{2}(z+u)^2\right\}\mathrm{d}z\right]\\
      &+ \int_{0}^{\infty}z^{\nu}(z - u) \exp\left\{-\frac{1}{2}(z-u)^2\right\}\mathrm{d}z - \int_{0}^{\infty}z^{\nu}(z + u) \exp\left\{-\frac{1}{2}(z+u)^2\right\}\mathrm{d}z
    \end{split}\\
    \label{eq:int.z.e.diff}
    \Leftrightarrow 0 &< \int_{0}^{\infty}z^{\nu+1}\exp\left\{-\frac{1}{2}(z-u)^2\right\}\mathrm{d}z - \int_{0}^{\infty}z^{\nu+1}\exp\left\{-\frac{1}{2}(z+u)^2\right\}\mathrm{d}z.
  \end{align}

  We can show that the difference in integrands in \eqref{eq:int.z.e.diff} is point-wise greater than 0 for all $z > 0$, so the difference of the integrals is greater than 0.
  \begin{align}
    \label{eq:z.e.pointwise}
    &z^{\nu+1}\exp\left\{-\frac{1}{2}(z-u)^2\right\} - z^{\nu+1}\exp\left\{-\frac{1}{2}(z+u)^2\right\} > 0\\
    \Leftrightarrow& \exp\left\{-\frac{1}{2}(z-u)^2\right\} > \exp\left\{-\frac{1}{2}(z+u)^2\right\}\\
    \Leftrightarrow& (z-u)^2 < (z+u)^2\\
    \Leftrightarrow& z^2 + u^2 - 2zu < z^2 + u^2 + 2zu\\
    \label{eq:z.u.times}
    \Leftrightarrow& -zu < zu.
  \end{align}
  Equation \eqref{eq:z.u.times} is true since both $z$ and $u$ are positive. Thus, integrating both sides of \eqref{eq:z.e.pointwise} indicates that \eqref{eq:int.z.e.diff} is positive, which indicates that the derivative \eqref{eq:fun.of.u} is positive, so \eqref{eq:fun.of.u} is monotone increasing in $u > 0$. Thus, \eqref{eq:log.t.diff.ratio} is monotone increasing in $t > 0$, indicating that $g(t)$ is monotone increasing in $|t|$ for any choice of symmetric $\pi(\cdot)$.
\end{proof}

\paragraph{Theorem~\ref{theo:teta}}
\begin{proof}
  From \eqref{eq:welford.mean} and \eqref{eq:welford.var}, we see that
  $\hat{\mu} = \bar{X} + \mathcal{O}_p\left(\frac{1}{n}\right)$ and
  $\hat{\sigma}^2 = S^2 + \mathcal{O}_p\left(\frac{1}{n}\right)$. Thus,
  \begin{align}
    \Pr\left( \frac{(\hat{\mu} - \mu)^2}{\hat{\sigma}^2 / (n+1)} \geq \eta^2 \right) &= \Pr\left( \frac{\left(\bar{X} + \mathcal{O}_p\left(\frac{1}{n}\right) - \mu\right)^2}{\left[S^2 + \mathcal{O}_p\left(\frac{1}{n}\right)\right] / (n+1)} \geq \eta^2 \right)\\
                                                                                     &= \Pr\left( \frac{\left(\bar{X} - \mu + \mathcal{O}_p\left(\frac{1}{n}\right)\right)^2}{S^2 / (n+1) + \mathcal{O}_p\left(\frac{1}{n^2}\right)} \geq \eta^2 \right)\\
    \label{eq:s2.np1.to.n}
                                                                                     &= \Pr\left( \frac{\left(\bar{X} - \mu + \mathcal{O}_p\left(\frac{1}{n}\right)\right)^2}{S^2 / n + \mathcal{O}_p\left(\frac{1}{n^2}\right)} \geq \eta^2 \right)\\
                                                                                     &= \Pr\left( \frac{\left(\frac{\bar{X} - \mu}{S/\sqrt{n}} + \mathcal{O}_p\left(\frac{1}{\sqrt{n}}\right)\right)^2}{1 + \mathcal{O}_p\left(\frac{1}{n}\right)} \geq \eta^2 \right)\\
                                                                                     &= \Pr\left( T^2\frac{\left(1 + \mathcal{O}_p\left(\frac{1}{\sqrt{n}}\right)\right)^2}{1 + \mathcal{O}_p\left(\frac{1}{n}\right)} \geq \eta^2 \right) ~\left(\text{where } T = \frac{\bar{X} - \mu}{S/\sqrt{n}}\right)\\
                                                                                     &= \Pr\left( T^2\left(1 + \mathcal{O}_p\left(\frac{1}{\sqrt{n}}\right)\right) \geq \eta^2 \right)\\
    \label{eq:t2.eta2.op}
                                                                                     &= \Pr\left( T^2 + \mathcal{O}_p\left(\frac{1}{\sqrt{n}}\right) \geq \eta^2 \right),
  \end{align}
  where \eqref{eq:s2.np1.to.n} uses the fact that
  $S^2/(n+1) = S^2/n + \mathcal{O}_p\left(\frac{1}{n^2}\right)$. Since $T$ follows a $t_{n-1}$ distribution, the result follows from \eqref{eq:t2.eta2.op}.
\end{proof}

\paragraph{Theorem~\ref{theo:squared.half.width}}
\begin{proof}
\begin{align}
  &\E\left[\eta_{\alpha}^2\left(\frac{n-1}{n}S^2 + \frac{1}{n+1}(\bar{X} - A)^2\right) / (n+1)\right]\\
  &=\frac{\eta_{\alpha}^2\sigma^2}{n(n+1)}\E\left[(n-1)S^2/\sigma^2\right] + \frac{\eta_{\alpha}^2\sigma^2}{n(n+1)^2}\E\left[\left(\frac{\bar{X} - \mu}{\sigma/\sqrt{n}} + \sqrt{n}\frac{\mu - A}{\sigma}\right)^2\right]\\
  &= \frac{n-1}{n(n+1)}\eta_{\alpha}^2\sigma^2 + \frac{1 + n\nu^2}{n(n+1)^2}\eta_{\alpha}^2\sigma^2\\
  &= \frac{n + \nu^2}{(n+1)^2}\eta_{\alpha}^2\sigma^2.
\end{align}
\end{proof}

\section{Conditions for marginal tail probability approximation}
\label{sec:marg.tail}

Some additional conditions on $\rho(\cdot)$ are needed for
\eqref{eq:marg.tail.prob.text} to hold for a given proper $\pi(\cdot)$. This is
since the remainder term in the Taylor Series
\eqref{eq:taylor.series.orig.proof} could blow up when integrating
over $\nu$. Theorem~\ref{theo:reg.cond.marg.theo} provides sufficient
(but not necessary) regularity conditions on $\rho(\cdot)$ for a given $\pi(\cdot)$.

\begin{theorem}
  \label{theo:reg.cond.marg.theo}
  Suppose that $\rho(\cdot)$ is symmetric, once differentiable, and for a
  given proper $\pi(\cdot)$,
  \begin{align}
    \label{eq:reg.cond.1}
    \int_{0}^{\infty}\nu\rho\left(\nu\right)\pi(\nu)\mathrm{d}\nu &< \infty,
  \end{align}
  and there exists a $z_0 > 0$ such that
  \begin{align}
    \label{eq:reg.cond.2}
    \int_{0}^{\infty}\nu^2M_{z_0}(\nu)\pi(\nu)\mathrm{d}\nu &< \infty, \text{ where},\\
    \label{eq:lipschitz}
    M_{z_0}(\nu) &= \sup_{|u+\nu| \leq \nu/z_0}|\rho'(u)|.
  \end{align}
  Then the marginal posterior tail probability of $\beta$ given $Y$ is
  \begin{align}
    \label{eq:marg.tail.prob}
    \Pr(\beta > z|Y) = \frac{|Y|}{z}\int_{0}^{\infty}\nu\rho\left(\nu\right)\pi(\nu)\mathrm{d}\nu + \mathcal{O}\left(\frac{1}{z^2}\right).
  \end{align}
\end{theorem}
\begin{proof}
  As in the proof of Theorem~\ref{theo:tail.prob}, we set $w = 1/z$
  and $t = 1/\beta$, then we do a Taylor series of the tail
  probability at $w=0$ (the same as $z = \infty$). This time, however,
  we use an explicit form for the remainder.
  \begin{align}
    \Pr(\beta > z|\nu,Y) &= \int_{z}^{\infty}\frac{\nu |Y|}{\beta^2}\rho\left(\nu\frac{Y}{\beta}-\nu\right)\mathrm{d}\beta\\
                         &= \int_0^{1/z}\nu |Y|\rho\left(t\nu Y-\nu\right)\mathrm{d}t\\
    \label{eq:expaneded.int.about.0}
    &= \frac{|Y|}{z}\nu\rho(\nu) +  |Y|\nu\int_0^{1/z}[\rho\left(t\nu Y-\nu\right) - \rho(-\nu)]\mathrm{d}t.
  \end{align}
  We will work to bound the remainder in
  \eqref{eq:expaneded.int.about.0}. By the mean value theorem, there
  exists a $\xi(t) \in [-\nu, t \nu Y - \nu]$ such that
  \begin{align}
    \label{eq:mvt.app}
    \rho(t \nu Y - \nu) - \rho(-\nu) = \nu Y t \rho'(\xi(t)).
  \end{align}
  The integrand in \eqref{eq:expaneded.int.about.0} has
  $t \in [0, 1/z]$. So, for such $t$, and for $z > z_0 |Y|$, we have
  \begin{align}
    |\xi(t) + \nu| \leq |\nu Y t| \leq \frac{|\nu Y|}{z} \leq \frac{\nu}{z_{0}}.
  \end{align}
  Thus, for large enough $z$, $\xi(t)$ is within the neighborhood of
  the supremum of \eqref{eq:lipschitz}. This means, for large enough
  $z$, the remainder in \eqref{eq:expaneded.int.about.0} is bounded by
  \begin{align}
    \left| |Y|\nu\int_0^{1/z}[\rho\left(t\nu Y-\nu\right) - \rho(-\nu)]\mathrm{d}t \right|&= Y^2\nu^2\left|\int_0^{1/z} t\rho'(\xi(t)) \mathrm{d}t \right|\\
                                                                                            &\leq Y^2\nu^2M_{z_0}(\nu)\int_0^{1/z} t \mathrm{d}t\\
                                                                                            &\leq \frac{Y^2}{z^2}\nu^2M_{z_0}(\nu).
  \end{align}

  We integrate \eqref{eq:expaneded.int.about.0} with respect to $\nu$,
  \begin{align}
    \label{eq:marg.prior.to.simplification}
    \Pr(\beta > z | Y) &= \frac{|Y|}{z}\int_{0}^{\infty}\nu\rho(\nu)\pi(\nu)\mathrm{d}\nu +   \int_{0}^{\infty}|Y|\nu\int_0^{1/z}[\rho\left(t\nu Y-\nu\right) - \rho(-\nu)]\mathrm{d}t\pi(\nu)\mathrm{d}\nu.
  \end{align}
  Regularity condition \eqref{eq:reg.cond.1} guarantees that the
  leading term in \eqref{eq:marg.prior.to.simplification} is
  finite. We will now bound the remainder in
  \eqref{eq:marg.prior.to.simplification}.
  \begin{align}
    \label{eq:bound.marg.remainder}
    &\left|\int_{0}^{\infty}|Y|\nu\int_0^{1/z}[\rho\left(t\nu Y-\nu\right) - \rho(-\nu)]\mathrm{d}t\pi(\nu)\mathrm{d}\nu\right| \leq \frac{Y^2}{z^2}\int_{0}^{\infty}\nu^2M_{z_0}(\nu)\pi(\nu)\mathrm{d}\nu.
  \end{align}
  Regularity condition \eqref{eq:reg.cond.2} guarantees that the
  remainder in \eqref{eq:bound.marg.remainder} is
  $\mathcal{O}\left(\frac{1}{z^2}\right)$. This finishes the proof of
  \eqref{eq:marg.tail.prob}.
\end{proof}

Regularity condition \eqref{eq:reg.cond.2} is rather complicated and
hard to verify. So we have Lemma~\ref{lemma:simp.reg} that contains a
simpler condition that implies \eqref{eq:reg.cond.2} for any proper
$\pi(\cdot)$.

\begin{lemma}
  \label{lemma:simp.reg}
  Suppose
  \begin{align}
    \label{eq:simpler.reg.cond.2}
    \sup_{u\in\mathbb{R}} u^2|\rho'(u)| &< \infty,
  \end{align}
  then regularity condition \eqref{eq:reg.cond.2} is satisfied for any
  proper $\pi(\cdot)$.
\end{lemma}
\begin{proof}
  For \eqref{eq:reg.cond.2}, we will bound
  $M_{z_0}(\nu) = \sup_{|u+\nu| \leq \nu/z_0}|\rho'(u)|$ for large
  enough $z_0$. Suppose $z_0 > 1$ and $|u+\nu| \leq \nu/z_0$. Then
  \begin{align}
    |u| &\geq \nu - |u + \nu| \text{ (reverse triangle inequality)}\\
        &\geq \nu - \nu/z_0 \text{ (neighborhood of supremum)}\\
        &=(1-1/z_0)\nu\\
    \label{eq:lower.bound.u}
        &=\alpha\nu,
  \end{align}
  for $\alpha = 1 - 1 / z_0 \in (0, 1)$.

  Thus, given $|u+\nu| \leq \nu/z_0$ and $z_0 > 1$, we have, for some finite $K$, and for $\nu \neq 0$,
  \begin{align}
    |\rho'(u)| &\leq \frac{K}{u^2} \text{ (from \eqref{eq:simpler.reg.cond.2})}\\
    \leq&\frac{K}{\alpha^2\nu^2} \text{ (from \eqref{eq:lower.bound.u})}.
  \end{align}
  Thus, for $\nu \neq 0$,
  \begin{align}
    M_{z_0}(\nu) &= \sup_{|u+\nu| \leq \nu/z_0}|\rho'(u)|\\
                 &\leq \frac{K}{\alpha^2\nu^2}\\
    \label{eq:bound.on.nu2.m}
    \Leftrightarrow & \nu^2M_{z_0}(\nu) \leq \frac{K}{\alpha^2} = \frac{K}{(1 - 1/z_0)^2}.
  \end{align}
  For $\nu = 0$, bound \eqref{eq:bound.on.nu2.m} is still satisfied because we assumed that $\rho(\cdot)$ is once differentiable, so $M_{z_0}(\nu) = |\rho'(0)| < \infty$.

  Getting back to regularity condition \eqref{eq:reg.cond.2},
  \begin{align}
    \int_{0}^{\infty}\nu^2M_{z_0}(\nu)\pi(\nu)\mathrm{d}\nu \leq \frac{K}{(1 - 1/z_0)^2}\int_{0}^{\infty}\pi(\nu)\mathrm{d}\nu = \frac{K}{(1 - 1/z_0)^2} < \infty.
  \end{align}
\end{proof}

We'll use Lemma~\ref{lemma:simp.reg} to prove that the normal
distribution satisfies \eqref{eq:marg.tail.prob.text} for all proper
$\pi(\cdot)$.

\begin{theorem}
  \label{theo:normal.reg.cond}
  Let $\rho(\cdot) = \phi(\cdot)$ be the standard normal probability density
  function. Then regularity conditions
  \eqref{eq:reg.cond.1}--\eqref{eq:reg.cond.2} are satisfied for any
  proper prior $\pi(\cdot)$ over $\nu$.
\end{theorem}
\begin{proof}
  Condition \eqref{eq:reg.cond.1} can be proven by noting that, in the normal case, $|\nu|\phi(\nu) \leq \phi(1)$ for all $\nu$. Thus
  \begin{align}
    \int_{0}^{\infty}|\nu|\phi\left(\nu\right)\pi(\nu)\mathrm{d}\nu \leq \phi(1)\int_{0}^{\infty}\pi(\nu)\mathrm{d}\nu = \phi(1) < \infty.
  \end{align}

  We will now use the Lemma~\ref{lemma:simp.reg} to prove regularity
  condition \eqref{eq:reg.cond.2}. For \eqref{eq:simpler.reg.cond.2},
  we have
  $x^2|\phi'(x)| = |x^3|\phi(x) \leq 3^{3/2}\phi(\sqrt{3}) < \infty$, so
  \eqref{eq:simpler.reg.cond.2} is satisfied.
\end{proof}

\section{Properties of inverse location-scale families}
\label{sec:inv.norm.properties}

Let $\rho(\cdot)$ be the standard PDF of a continuous symmetric
location-scale family (e.g., standard normal). Suppose $1/Z$ belongs
to this family with location $a$ and scale $b$. Then the PDF of $Z$ is
(through a simple change of variables)
\begin{align}
  \frac{1}{bz^2}\rho\left(\frac{1/z - a}{b}\right).
\end{align}
In this section, we derive the CDF and quantile function for the distribution of $1/Z$, which we call an inverse location-scale distribution.

Let $F(\cdot)$ be the CDF of the standard distribution in the
location-scale family (e.g., in the normal case $\Phi(\cdot)$). The
CDF, $\Pr(Z \leq z)$ can be found considering two cases
(Figure~\ref{fig:inv.norm.bounds}). First, assume $z \leq 0$, then
\begin{align}
  \label{eq:inv.norm.bound.1}
  \Pr(Z \leq z) &= \Pr\left(\frac{1}{z} \leq \frac{1}{Z} \leq 0\right)\\
                &= \Pr\left(\frac{1/z - a}{b} \leq \frac{1/Z - a}{b} \leq -\frac{a}{b}\right)\\
                &= F\left(-\frac{a}{b}\right) - F\left(\frac{1/z - a}{b}\right)\\
  \label{eq:cdf.inv.norm.negz}
                &= F\left(\frac{az - 1}{bz}\right) - F\left(\frac{a}{b}\right).
\end{align}
Second, assume $z \geq 0$, then
\begin{align}
  \label{eq:inv.norm.bound.2}
  \Pr(Z \leq z) &= \Pr\left(\frac{1}{Z} \leq 0\right) + \Pr\left(\frac{1}{z} \leq \frac{1}{Z}\right)\\
                &= \Pr\left(\frac{1/Z - a}{b} \leq -\frac{a}{b}\right) + \Pr\left(\frac{1/z - a}{b} \leq \frac{1/Z - a}{b}\right)\\
                &= F\left(-\frac{a}{b}\right) + 1 - F\left(\frac{1/z - a}{b}\right)\\
  \label{eq:cdf.inv.norm.posz}
                &= 1 + F\left(\frac{az - 1}{bz}\right) - F\left(\frac{a}{b}\right).
\end{align}

\begin{figure}[!htb]
  \centering
  \includegraphics{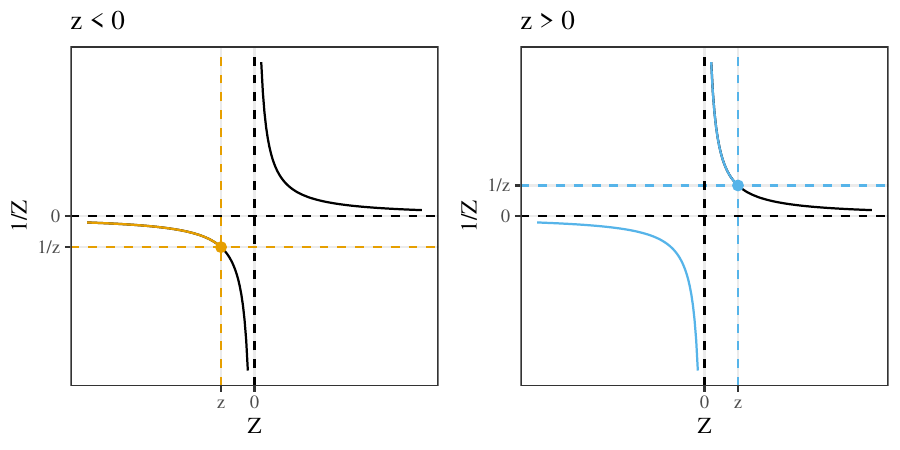}
  \caption{Visualization of the calculations in
    \eqref{eq:inv.norm.bound.1} and \eqref{eq:inv.norm.bound.2}. In
    the left facet, where $z < 0$, we have $Z \leq z$ if and only if
    $1/z \leq 1 / Z \leq 0$. In the right facet, where $z > 0$, we
    have $Z\leq z$ if and only if either $1/Z < 0$ or $1/z < 1/Z$.}
  \label{fig:inv.norm.bounds}
\end{figure}

Let $F^{-1}(\cdot)$ be the quantile function of the standard
distribution in the location-scale family (e.g., in the normal case
$\Phi^{-1}(\cdot)$). For the quantile function, let $p = \Pr(Z \leq z)$
taking $z \rightarrow 0$ from the left in \eqref{eq:cdf.inv.norm.negz}
or from the right in \eqref{eq:cdf.inv.norm.posz} shows that $z < 0$
if and only if $p < 1 - F\left(\frac{a}{b}\right)$. Solving for $z$
in \eqref{eq:cdf.inv.norm.negz} and \eqref{eq:cdf.inv.norm.posz} gives
us the quantile function
\begin{align}
  \label{eq:quant.invnorm}
  f(p) =
  \begin{cases}
    \left[a - bF^{-1}\left(p + F\left(\frac{a}{b}\right)\right)\right]^{-1} & \text{ if } p \leq 1 - F\left(\frac{a}{b}\right)\\
    \left[a - bF^{-1}\left(p + F\left(\frac{a}{b}\right) - 1\right)\right]^{-1} & \text{ if } p > 1 - F\left(\frac{a}{b}\right).
  \end{cases}
\end{align}

\section{Inverted posterior density ratio test approaches}
\label{sec:pdr}

Note that $(X + A)/2$ and $|X-A|/\sqrt{2}$ are the sample mean and
sample standard deviation of $X$ and $A$. Interval
\eqref{eq:ci.init.2} can therefore be seen as an ``augmented''
$t$-interval using data $X$ and $A$ (mean plus or minus some multiple
of the sample standard deviation). The multiplier is determined by
bounding the error probability below $\alpha$.

Using this insight, there is a Bayesian interpretation of
\eqref{eq:ci.init.2} in terms of posterior density ratios
\citep{basu1996bayesian}. An interval from the posterior density ratio
is of the form
\begin{align}
  \label{eq:pdr.profile}
  \mu : \frac{\max_{\sigma^2}f(X|\mu,\sigma^2)\pi(\mu,\sigma^2)}{\max_{\mu,\sigma^2}f(X|\mu,\sigma^2)\pi(\mu,\sigma^2)} > \xi,
\end{align}
for some $\xi$ chosen to control the error probability. Plugging in
the normal likelihood and using prior
\begin{align}
  \label{eq:prior.pdr}
  \pi(\mu, \sigma^2) = \N(\mu|A, \sigma^2),
\end{align}
we get
\begin{align}
  &\mu : \frac{\max_{\sigma^2}\N(X|\mu,\sigma^2)\N(\mu|A,\sigma^2)}{\max_{\mu,\sigma^2}\N(X|\mu,\sigma^2)\N(\mu|A,\sigma^2)} > \xi\\
  \label{eq:pdr.form}
  \Leftrightarrow &\mu : \frac{\max_{\sigma^2}\N(X|\mu,\sigma^2)\N(A|\mu,\sigma^2)}{\max_{\mu,\sigma^2}\N(X|\mu,\sigma^2)\N(A|\mu,\sigma^2)} > \xi.
\end{align}
Equation \eqref{eq:pdr.form} is the same as a likelihood ratio test
statistic using augmented data $X$ and $A$. In the normal case, the
likelihood ratio test is the same as the $t$-test \citep[Section
8.2]{casella2024statistical}. Thus, intervals from PDRs using prior
\eqref{eq:prior.pdr} yield intervals that are exactly of the form
\eqref{eq:ci.init.2}.

Instead of using a ``profile posterior'' approach as in
\eqref{eq:pdr.profile}, we could also take a marginal posterior
approach by first integrating out $\sigma^2$.
\begin{align}
  \label{eq:pdr.profile.marginal}
  \mu : \frac{\int f(X|\mu,\sigma^2)\pi(\mu,\sigma^2)\mathrm{d}\sigma^2}{\max_{\mu}\int f(X|\mu,\sigma^2)\pi(\mu,\sigma^2)\mathrm{d}\sigma^2} > \xi,
\end{align}
for, again, some $\xi$ chosen to control the error probability. For
the marginal approach, we use priors of the form:
\begin{align}
  \label{eq:prior.pdr.marginal}
  \pi(\mu, \sigma^2) = \N(\mu|A, \sigma^2)(\sigma^2)^{-1}.
\end{align}
The complete likelihood is then
\begin{align}
  \pi(X, \mu, \sigma^2) &= \N(X|\mu, \sigma^2) \N(\mu|A,\sigma^2) (\sigma^2)^{-1}\\
  \label{eq:complete.likelihood}
                        &= \N(X|\mu, \sigma^2) \N(A|\mu,\sigma^2) (\sigma^2)^{-1}.
\end{align}
Equation \eqref{eq:complete.likelihood} can be seen, from a different
perspective, as a normal model with augmented data $(X,A)$ and prior
$\pi(\mu, \sigma^2) = (\sigma^2)^{-1}$. It is well known that the
marginal posterior for $\mu$ in this case is Cauchy with location
$(X + A)/2$ and scale $|X - A|/2$ \citep[Section
3.2]{gelman2013bayesian}. Thus, using prior
\eqref{eq:prior.pdr.marginal} and marginal PDR interval
\eqref{eq:pdr.profile.marginal}, we again get intervals of the same
form as the $t$-intervals using augmented data $(X,A)$, and thus
intervals of the form \eqref{eq:ci.init.2}.

This marginal approach extends naturally to larger sample
sizes. Suppose $X_1,\ldots,X_n \overset{\text{iid}}{\sim} \N(\mu, \sigma^2)$ and we have
prior $\pi(\mu, \sigma^2) = \N(\mu|A, \sigma^2)(\sigma^2)^{-1}$. Then
this can be viewed, from another perspective, as using augmented data
$X_1,\ldots,X_n,A \overset{\text{iid}}{\sim} \N(\mu,\sigma^2)$ and prior
$\pi(\mu,\sigma^2) = (\sigma^2)^{-1}$. Well-known results
\citep[Section 3.2]{gelman2013bayesian} are that the marginal
posterior of $\mu$ is
\begin{align}
  T := \frac{\mu - \hat{\mu}}{\hat{\sigma} / \sqrt{n+1}} \sim t_{n},
\end{align}
where $\hat{\mu}$ and $\hat{\sigma}$ are the sample mean and sample
standard deviation of the augmented data $X_1,\ldots,X_n,A$. The
marginal PDR intervals \eqref{eq:pdr.profile.marginal} are then of the form
\begin{align}
  &(1 + T^2 / n)^{-(n+1)/2} > c_1\\
  &\Leftrightarrow T^2 < c_2\\
  &\Leftrightarrow  \left(\frac{\mu - \hat{\mu}}{\hat{\sigma} / \sqrt{n+1}}\right)^2 < c_2\\
  \label{eq:aug.t.int.from.marg.pdr}
  &\Leftrightarrow \mu \in \hat{\mu} \pm c_3 \hat{\sigma}/\sqrt{n+1},
\end{align}
for $c_i$ some constants that control the error probability. Equation
\eqref{eq:aug.t.int.from.marg.pdr} is exactly the form of the
augmented $t$-intervals.

\section{Bayes factors and the augmented $t$-test}
\label{sec:bf.aug.t}

Let $X_1,X_2,\ldots,X_n \overset{\text{iid}}{\sim} N(\mu,\sigma^2)$. We are testing
$H_0: \mu = \mu_0$ versus $H_A: \mu \neq \mu_0$. Let
$\delta = (\mu - \mu_0) / \sigma$. Consider the collection of priors
\begin{align}
  \label{eq:prior.h0}
  \pi(\sigma^2|H_0) &= 1/\sigma^2 \text{ and}\\
  \label{eq:prior.ha}
  \pi(\delta,\sigma^2|H_A) &= \pi(\delta) / \sigma^2.
\end{align}
In this section, we show that the Bayes factor calculated using priors \eqref{eq:prior.h0} and \eqref{eq:prior.ha},
\begin{align}
  \label{eq:bf.x}
  \mathrm{BF} = \frac{\int_{\delta,\sigma^2}\frac{\pi(\delta)}{\sigma^2}\prod_{i=1}^nN(X_i|\mu_0 + \delta\sigma, \sigma^2) \mathrm{d}\sigma^2 \mathrm{d}\delta}{\int_{\sigma^2}\frac{1}{\sigma^2}\prod_{i=1}^nN(X_i|\mu_0, \sigma^2) \mathrm{d}\sigma^2}
\end{align}
is a function of the $t$-statistic,
$t = \frac{\bar{X} - \mu_0}{S / \sqrt{n}}$, where $\bar{X}$ and $S$
are the sample mean and sample standard deviation, respectively. The
results of this section extend to the augmented $t$-statistic results
in Section~\ref{sec:bf.invert} by noting that using priors
\eqref{eq:prior.h0.aug} and \eqref{eq:prior.ha.aug} is equivalent to
the current setup but using data $X_1,X_2,\ldots,X_n,A$. Though, we
also must assume that $A \neq \mu_0$ (a Lebesgue measure 0 event).

Consider the transformed problem
$Y_i = X_i - \mu_0 \sim N(\mu - \mu_0, \sigma^2)$. Let
$\delta = \frac{\mu - \mu_0}{\sigma}$ be the standardized effect size
so that we can write $Y_i \sim N(\sigma\delta, \sigma^2)$. Then we are
equivalently testing $H_0: \delta = 0$ versus $H_A: \delta \neq
0$. The Bayes factor in \eqref{eq:bf.x} is equivalent, using the $Y_i$'s, to
\begin{align}
  \label{eq:bf.y}
  \mathrm{BF} = \frac{\int_{\delta,\sigma^2}\frac{\pi(\delta)}{\sigma^2}\prod_{i=1}^nN(Y_i|\delta\sigma, \sigma^2) \mathrm{d}\sigma^2 \mathrm{d}\delta}{\int_{\sigma^2}\frac{1}{\sigma^2}\prod_{i=1}^nN(Y_i|0, \sigma^2) \mathrm{d}\sigma^2}.
\end{align}

We begin with showing that the likelihood is a function of the sample
mean and sample variance Lemma~\ref{lem:like.y}.
\begin{lemma}
  \label{lem:like.y}
  The likelihood for the data is
  \begin{align}
    f(Y_1,\ldots,Y_n|\delta,\sigma^2) &= \prod_{i=1}^nN(Y_i|\delta\sigma, \sigma^2)\\
                                      &= (2\pi\sigma^2)^{-n/2}\exp\left\{-\frac{1}{2\sigma^2}\left[(n-1)S^2 + n(\bar{Y} - \delta\sigma)^2\right]\right\},
  \end{align}
  where $\bar{Y} = \bar{X} - \mu_0$ is the sample mean of the $Y_i$'s, and $S^2$ is the sample variance of the $Y_i$'s (and $X_i$'s).
\end{lemma}
\begin{proof}
  This is very well known.
  \begin{align}
    f&(Y_1,\ldots,Y_n|\delta,\sigma^2) =\prod_{i=1}^nN(Y_i|\delta\sigma, \sigma^2)\\
     &= (2\pi\sigma^2)^{-n/2}\exp\left\{-\frac{1}{2\sigma^2}\sum_{i=1}^n(Y_i - \delta\sigma)^2\right\}\\
     &= (2\pi\sigma^2)^{-n/2}\exp\left\{-\frac{1}{2\sigma^2}\sum_{i=1}^n(Y_i - \bar{Y} + \bar{Y} -  \delta\sigma)^2\right\}\\
     &= (2\pi\sigma^2)^{-n/2}\exp\left\{-\frac{1}{2\sigma^2}\left[\sum_{i=1}^n(Y_i - \bar{Y})^2 + 2(\bar{Y} -  \delta\sigma)\sum_{i=1}^n(Y_i - \bar{Y}) + n (\bar{Y} -  \delta\sigma)^2\right]\right\}\\
    &= (2\pi\sigma^2)^{-n/2}\exp\left\{-\frac{1}{2\sigma^2}\left[(n-1)S^2 + n(\bar{Y} -  \delta\sigma)^2\right]\right\}.
  \end{align}
\end{proof}

We will now show that the denominator in \eqref{eq:bf.y} is proportional to a $t$-density with $n-1$ degrees of freedom.
\begin{lemma}
  \label{lem:t.integral}
  \begin{align}
    \label{eq:t.integral}
    \int_{\sigma^2}\frac{1}{\sigma^2}\prod_{i=1}^nN(Y_i|0, \sigma^2) \mathrm{d}\sigma^2 = [(n-1)\pi]^{-\frac{n-1}{2}}[S^2]^{-\frac{n}{2}}\Gamma\left(\frac{n-1}{2}\right)T_{n-1}(t),
  \end{align}
  where $T_{n-1}(\cdot)$ is the $t$-density with $n-1$ degrees of freedom
  and
  $t = \frac{\bar{Y}}{S/\sqrt{n}} = \frac{\bar{X} -
    \mu_0}{S/\sqrt{n}}$ is the $t$-statistic for testing
  $H_0: \mu = \mu_0$ versus $H_A: \mu \neq \mu_0$.
\end{lemma}
\begin{proof}
  From Lemma~\ref{lem:like.y}, we have
  \begin{align}
    \label{eq:inv.gamma.kern.h0.a}
    \int_{\sigma^2}&\frac{1}{\sigma^2}\prod_{i=1}^nN(Y_i|0, \sigma^2) \mathrm{d}\sigma^2 = \int_{\sigma^2}(2\pi)^{-n/2}(\sigma^2)^{-n/2-1}\exp\left\{-\frac{1}{2\sigma^2}\left[(n-1)S^2 + n\bar{Y}^2\right]\right\}\mathrm{d}\sigma^2\\
    \label{eq:inv.gamma.kern.h0.b}
                   &=(2\pi)^{-n/2}\Gamma\left(\frac{n}{2}\right)2^{n/2}\left[(n-1)S^2 + n\bar{Y}^2\right]^{-n/2}\\
                   &=\Gamma\left(\frac{n}{2}\right)[(n-1)\pi S^2]^{-n/2}\left[1 + \left(\frac{\bar{Y}}{S / \sqrt{n}}\right)^2 / (n-1)\right]^{-n/2}\\
                   &=[(n-1)\pi S^2]^{-n/2} \sqrt{\pi(n-1)}\Gamma\left(\frac{n-1}{2}\right) \frac{\Gamma\left(\frac{(n-1) + 1}{2}\right)}{\sqrt{\pi(n-1)}\Gamma\left(\frac{n-1}{2}\right)} \left[1 +  \frac{t^2}{n-1}\right]^{-\frac{(n-1) + 1}{2}}\\
                   &= [(n-1)\pi]^{-\frac{n-1}{2}}[S^2]^{-\frac{n}{2}}\Gamma\left(\frac{n-1}{2}\right)T_{n-1}(t).
  \end{align}
  Line \eqref{eq:inv.gamma.kern.h0.b} follows since the integrand in
  \eqref{eq:inv.gamma.kern.h0.a} is the kernel of an inverse-gamma
  distribution with shape $\frac{n}{2}$ and scale
  $\frac{1}{2}\left[(n-1)S^2 + n\bar{Y}^2\right]$
\end{proof}

We will now show that the numerator in \eqref{eq:bf.y} is proportional
to an integral over $\delta$ of a non-central $t$-density with $n-1$
degrees of freedom and non-centrality parameter $\sqrt{n}\delta$.
\begin{lemma}
  \label{lem:noncentral.t.integral}
  \begin{align}
    \label{eq:noncentral.t.integral}
    \int_{\sigma^2}\frac{1}{\sigma^2}\prod_{i=1}^nN(Y_i|\delta\sigma, \sigma^2) \mathrm{d}\sigma^2 = [(n-1)\pi]^{-\frac{n-1}{2}}[S^2]^{-\frac{n}{2}}\Gamma\left(\frac{n-1}{2}\right) T_{n-1}(t|\sqrt{n}\delta),
  \end{align}
  where $T_{n-1}(t|\sqrt{n}\delta)$ is the non-central $t$-density
  with $n-1$ degrees of freedom and non-centrality parameter
  $\sqrt{n}\delta$, and
  $t = \frac{\bar{Y}}{S/\sqrt{n}} = \frac{\bar{X} -
    \mu_0}{S/\sqrt{n}}$ is the $t$-statistic for testing
  $H_0: \mu = \mu_0$ versus $H_A: \mu \neq \mu_0$.
\end{lemma}
\begin{proof}
  The random variable $T$ follows a non-central $t$-distribution with
  $\nu$ degrees of freedom and non-centrality parameter $\xi$ if
  $T = \frac{Z + \xi}{\sqrt{W^2/\nu}}$ where $Z$ is standard normal
  and $W^2$ is $\chi^2$ with $\nu$ degrees of freedom. Let
  $V^2 = 1 / W^2$, then this is equivalent to the mixture model
  $T|W^2 \sim N(\xi\sqrt{\nu}V, \nu V^2)$ and
  $V^2 \sim \text{Inv-}\chi^2_{\nu}$. We will show that the integral
  in \eqref{eq:noncentral.t.integral} is proportional to an inverse
  chi-squared mixture of $N(t|\sqrt{n}\delta \sqrt{n-1}V, (n-1)V^2)$
  densities with the provided proportionality constant.

  Let $v^2 = \frac{\sigma^2}{(n-1)S^2}$. Furthermore, let
  $\chi^{-2}_{n-1}(v^2)$ be the density of the inverse chi-squared
  distribution with $n-1$ degrees of freedom. From
  Lemma~\ref{lem:like.y} and a change of variables, we have
  \begin{align}
    \begin{split}
      \int_{\sigma^2}&\frac{1}{\sigma^2}\prod_{i=1}^nN(Y_i|\delta\sigma, \sigma^2) \mathrm{d}\sigma^2 \\
      &= \int_{\sigma^2}(2\pi)^{-n/2}(\sigma^2)^{-n/2 - 1}\exp\left\{-\frac{1}{2\sigma^2}\left[(n-1)S^2 + n(\bar{Y} - \delta\sigma)^2\right]\right\}\mathrm{d}\sigma^2
    \end{split}\\
      &= \int_{v^2}(2\pi (n-1)S^2)^{-n/2}(v^2)^{-n/2 - 1}\exp\left\{-\frac{1}{2v^2} -\frac{n(\bar{Y} - \delta\sqrt{n-1}Sv)^2}{2(n-1)S^2v^2}\right\}\mathrm{d}v^2\\
      &= \int_{v^2}(2\pi (n-1)S^2)^{-n/2}(v^2)^{-n/2 - 1}\exp\left\{-\frac{1}{2v^2} -\frac{n(\bar{Y} - \delta\sqrt{n-1}Sv)^2}{2(n-1)S^2v^2}\right\}\frac{\chi^{-2}_{n-1}(v^2)}{\chi^{-2}_{n-1}(v^2)}\mathrm{d}v^2\\
    \begin{split}
      &= \int_{v^2}(2\pi (n-1)S^2)^{-n/2}(v^2)^{-n/2 - 1}\exp\left\{-\frac{1}{2v^2} -\frac{n(\bar{Y} - \delta\sqrt{n-1}Sv)^2}{2(n-1)S^2v^2}\right\}\\
      &\phantom{=\int_{v^2}}\times 2^{\frac{n-1}{2}}\Gamma\left(\frac{n-1}{2}\right) (v^2)^{\frac{n-1}{2} + 1}\exp\left\{\frac{1}{2v^2}\right\}\chi^{-2}_{n-1}(v^2)\mathrm{d}v^2
    \end{split}\\
    \begin{split}
      &= [(n-1)\pi]^{-\frac{n-1}{2}}[S^2]^{-\frac{n}{2}}\Gamma\left(\frac{n-1}{2}\right)\\
      &\times \int_{v^2}\frac{1}{\sqrt{2\pi(n-1)v^2}}\exp\left\{-\frac{n(\bar{Y} - \delta\sqrt{n-1}Sv)^2}{2(n-1)S^2v^2}\right\}\chi^{-2}_{n-1}(v^2)\mathrm{d}v^2
    \end{split}\\
    \begin{split}
      &= [(n-1)\pi]^{-\frac{n-1}{2}}[S^2]^{-\frac{n}{2}}\Gamma\left(\frac{n-1}{2}\right)\\
      &\times \int_{v^2}\frac{1}{\sqrt{2\pi(n-1)v^2}}\exp\left\{-\frac{(\frac{\bar{Y}}{S/\sqrt{n}} - \sqrt{n}\delta\sqrt{n-1}v)^2}{2(n-1)v^2}\right\}\chi^{-2}_{n-1}(v^2)\mathrm{d}v^2
    \end{split}\\
    \label{eq:inv.chis.mix}
    \begin{split}
      &= [(n-1)\pi]^{-\frac{n-1}{2}}[S^2]^{-\frac{n}{2}}\Gamma\left(\frac{n-1}{2}\right)\\
      &\times \int_{v^2}\frac{1}{\sqrt{2\pi(n-1)v^2}}\exp\left\{-\frac{(t - \sqrt{n}\delta\sqrt{n-1}v)^2}{2(n-1)v^2}\right\}\chi^{-2}_{n-1}(v^2)\mathrm{d}v^2.
    \end{split}
  \end{align}
  From inspection \eqref{eq:inv.chis.mix} is the marginal density of
  $T$ where $T|V^2 \sim N(\sqrt{n}\delta \sqrt{n-1}V, (n-1)V^2)$ and
  $V^2 \sim \mathrm{Inv-}\chi^2_{n-1}$.
\end{proof}

\begin{theorem}
  The Bayes factor in \eqref{eq:bf.y} is equal to
  \begin{align}
    \mathrm{BF} = \frac{\int_{\delta}T_{n-1}(t|\sqrt{n}\delta)\pi(\delta)\mathrm{d}\delta}{T_{n-1}(t)}.
  \end{align}
\end{theorem}
\begin{proof}
  Take \eqref{eq:bf.y} and plug in \eqref{eq:noncentral.t.integral}
  from Lemma~\ref{lem:noncentral.t.integral} in the numerator and
  \eqref{eq:t.integral} from Lemma~\ref{lem:t.integral} in the
  denominator.
\end{proof}

\section{More augmentation does not help}
\label{sec:more.m}

We could have derived a more general form for the augmented
$t$-interval where we use $m$ replicates of $A$. That is, we use
augmented data
\begin{align}
  \label{eq:aug.dat.2.2}
  X_1,\ldots,X_n,\underbrace{A,\ldots,A}_{m \text{ times}},
\end{align}
though, $m$ need not be an integer.  Letting $\hat{\mu}$ and
$\hat{\sigma}^2$ be the sample mean and sample variance using this
augmented data \eqref{eq:aug.dat.2.2}, the augmented $t$-interval is
\begin{align}
  \label{eq:aug.t.int.2.2}
  \hat{\mu} \pm \eta \hat{\sigma}/\sqrt{n + m},
\end{align}
where $\eta$ is chosen to control the error probability. Increasing
$m > 1$ turns out to only improve augmented $t$-intervals for small
values of $\nu = (\mu - A)/\sigma$ when $n = 2$, but appears to
otherwise be dominated by the case when $m = 1$ for all $n > 2$. We
will explore these intervals in this section.

We find $\eta$ by Theorem~\ref{theo:form.aug.t.n.2}.
\begin{theorem}
  \label{theo:form.aug.t.n.2}
  \begin{gather}
    \frac{(\hat{\mu} - \mu)^2}{\hat{\sigma}^2/(n + m)} = \frac{m^2(n+m-1)}{n}\frac{(\frac{n}{m}Z - \sqrt{n}\nu)^2}{(n+m)W^2 + m(Z + \sqrt{n}\nu)^2}, \text{ where}\\
    Z = \frac{\bar{X} - \mu}{\sigma/\sqrt{n}} \sim \N(0,1),~  W^2 =  \frac{(n-1)S^2}{\sigma^2} \sim \chi^2_{n-1}, \text{ and } \nu = \frac{\mu - A}{\sigma},
  \end{gather}
  and $Z$ and $W^2$ are independent.
\end{theorem}
\begin{proof}
  From a generalization of \citet{welford1962note}, we can re-write
  $\hat{\mu}$ and $\hat{\sigma}^2$ as
  \begin{align}
    \label{eq:welford.mean}
    \hat{\mu} &= \frac{n\bar{X} + mA}{n + m}, \text{ and}\\
    \label{eq:welford.var}
    \hat{\sigma}^2 &= \frac{n-1}{n + m - 1}S^2 + \frac{nm}{(n + m)(n+ m - 1)}(\bar{X} - A)^2.
  \end{align}
  Thus, we have
  \begin{align}
    \frac{(\hat{\mu} - \mu)^2}{\hat{\sigma}^2/(n + m)} &= \frac{\left(\frac{n\bar{X} + mA}{n + m} - \mu\right)^2}{\frac{1}{(n + m)(n + m - 1)}\left((n-1)S^2 + \frac{nm}{(n + m)}(\bar{X} - A)^2\right)}\\
                                                       &= \frac{\frac{m^2}{(n+m)^2}\left(\frac{n}{m}(\bar{X} - \mu) - (\mu - A)\right)^2}{\frac{1}{(n + m)(n + m - 1)}\left((n-1)S^2 + \frac{nm}{(n + m)}((\bar{X}-\mu) + (\mu - A))^2\right)}\\
                                                       &= \frac{m^2(n + m - 1)}{n(n + m)}\frac{\left(\frac{n}{m}\frac{\bar{X} - \mu}{\sigma/\sqrt{n}} - \sqrt{n}\frac{\mu - A}{\sigma}\right)^2}{\frac{(n-1)S^2}{\sigma^2} + \frac{m}{(n + m)}(\frac{\bar{X}-\mu}{\sigma / \sqrt{n}} + \sqrt{n} \frac{\mu - A}{\sigma})^2}\\
                                                       &= \frac{m^2(n + m - 1)}{n}\frac{\left(\frac{n}{m}Z - \sqrt{n}\nu\right)^2}{(n + m)W^2 + m(Z + \sqrt{n}\nu)^2}.
  \end{align}
  The distributions of $Z$ and $W^2$ and their independence follow from elementary statistics texts.
\end{proof}

From Theorem~\ref{theo:form.aug.t.n.2}, for a given $\nu$ and $\eta$, the
coverage failure probability of interval \eqref{eq:aug.t.int.2.2} is
\begin{align}
  \label{eq:alpha.n2.failure.def.2}
  \alpha(\nu,\eta) := \Pr_{\nu}\left[\frac{m^2(n+m-1)}{n}\frac{(\frac{n}{m}Z - \sqrt{n}\nu)^2}{(n+m)W^2 + m(Z + \sqrt{n}\nu)^2} > \eta^2\right].
\end{align}
Thus, to find $\eta$, we maximize \eqref{eq:alpha.n2.failure.def.2} over
$\nu$, which results in a worst-case $\alpha$ as a function of
$\eta$. We can then invert this function to obtain the value of $\eta$
given a worst-case $\alpha$, say $\eta_{\alpha}$. Since
\eqref{eq:alpha.n2.failure.def.2} is a double integral, it is possible
to calculate numerically. Specifically, let
\begin{align}
  g(Z,\nu,\eta) = \frac{1}{\eta^2}\frac{m^2(n+m-1)}{n(n+m)}\left(\frac{n}{m}Z - \sqrt{n}\nu\right)^2 - \frac{m}{n+m}(Z + \sqrt{n}\nu)^2,
\end{align}
then
\begin{align}
  \alpha(\nu,\eta) = \int_{-\infty}^{\infty}\int_{0}^{g(z,\nu,\eta)}\chi^2_{n-1}(w^2)\mathrm{d}w^2\phi(z)\mathrm{d}z,
\end{align}
where $\phi(\cdot)$ is the standard normal density and
$\chi^2_{n-1}(\cdot)$ is the $\chi^2_{n-1}$ density.

We can calculate the expected squared half-widths of these intervals by the following
\begin{align}
  &\E\left[\eta_{\alpha}^2\left(\frac{n-1}{n + m - 1}S^2 + \frac{nm}{(n+m)(n + m - 1)}(\bar{X} - A)^2\right) / (n+m)\right]\\
  &=\frac{\eta_{\alpha}^2\sigma^2}{(n+m)(n + m - 1)}\E\left[\frac{(n-1)S^2}{\sigma^2}\right] + \frac{\eta_{\alpha}^2\sigma^2m}{(n+m)^2(n + m - 1)}\E\left[\left(\frac{\bar{X} - \mu}{\sigma/\sqrt{n}} + \sqrt{n}\frac{\mu - A}{\sigma}\right)^2\right]\\
  &= \frac{\eta_{\alpha}^2\sigma^2(n-1)}{(n+m)(n + m - 1)} + \frac{\eta_{\alpha}^2\sigma^2m(1 + n\nu^2)}{(n+m)^2(n + m - 1)}\\
  \label{eq:e.hw.augt.2}
  &= \frac{n}{(n+m)(n+m-1)}\left(1 + \frac{m\nu^2 - 1}{n+m}\right)\eta_{\alpha}^2\sigma^2,
\end{align}
where $\nu = (\mu - A) / \sigma$.

Inspecting the expected squared half widths of the more generalized
augmented $t$-interval \eqref{eq:e.hw.augt.2} and the standard
$t$-interval \eqref{eq:e.hw.studt}, we see that each is some
multiplier of the true variance. We plot this multiplier for the
standard $t$-interval and the augmented $t$-interval for $m = 1,2,3$
in Figure~\ref{fig:half.width.largerm}. There is always a region where
the augmented $t$-interval beats the standard $t$-interval when
$m = 1$. But $m > 1$ augmented $t$-intervals only have some regions of
improvement over $m = 1$ when $n = 2$, and do not improve over the
standard $t$-intervals for large enough $m$.

\begin{figure}[!htb]
  \centering
  \includegraphics{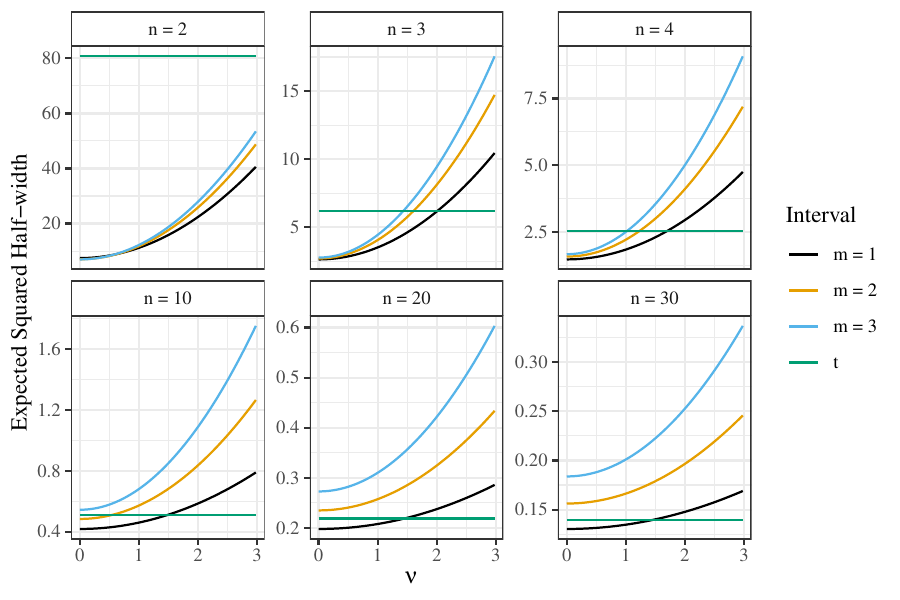}
  \caption{Multiplier to $\sigma^2$ to obtain the expected squared
    half-width of the standard $t$-interval \eqref{eq:e.hw.studt} or
    the augmented $t$-interval \eqref{eq:e.hw.augt.2} at different
    values of $m$. Values are plotted for different values of
    $\nu = (\mu - A) / \sigma$ and the sample size $n$. Smaller
    multiplier values indicate smaller expected squared half-width.}
  \label{fig:half.width.largerm}
\end{figure}

In this section, we used replicates of a single value $A$ to augment
our data. It would have been possible to derive more general augmented
$t$-intervals with augmented data $A_1,A_2,\ldots,A_m$ summarized by
sample mean $\bar{A}$ and sample variance $S_A^2$. However, doing so
would require an optimization over two parameters, $S_A^2/\sigma^2$
and $(\mu - \bar{A})/\sigma$ in \eqref{eq:alpha.n2.failure.def.2} to
find the worst-case $\alpha$ (and, hence, the corresponding
$\eta$). It is simpler, and more directly generalizable, to instead
augment with just $m$ copies of $A$. Thus, $S_A^2 = 0$ and
$\bar{A} = A$ and we need only optimize over
$\nu = (\mu - A) / \sigma$, similar to the $n = 1$ and $m = 1$ case of
Appendix~\ref{sec:bm.deriv}.

\end{document}